\documentclass[11pt]{amsart}
\usepackage{geometry, color}
\usepackage{graphicx}
\usepackage{amssymb}
\usepackage{amsmath}
\usepackage{amsthm}

\newtheorem{lemma}{Lemma}
\newtheorem{theorem}{Theorem}
\newtheorem{cor}{Corollary}
\newtheorem{pro}{Proposition}
\newtheorem{remark}{Remark}
\newcommand{\R}{\mathbb R}
\numberwithin{equation}{section}

\numberwithin{equation}{section}
\numberwithin{lemma}{section}
\numberwithin{theorem}{section}
\numberwithin{pro}{section}
\title[Small sphere limit]{Evaluating small sphere limit of the Wang--Yau quasi-local energy}
\author{PoNing Chen, Mu-Tao Wang, and Shing-Tung Yau}
\date{October 3, 2015}

\begin{document}
\thanks{Part of this work was carried out while all three authors were visiting the Department of Mathematics of National Taiwan University and Taida Institute for Mathematical Sciences in Taipei, Taiwan. Part of this work was carried out
when P.-N. Chen and M.-T. Wang were visiting the Department of Mathematics and the Center of Mathematical Sciences and Applications at Harvard  University. P.-N. Chen is supported by NSF grant DMS-1308164, M.-T. Wang is supported by NSF grants DMS-1105483 and DMS-1405152, and S.-T. Yau is supported by NSF grant  DMS-0804454 and DMS-1306313. This work was partially supported by a grant from the Simons Foundation (\#305519 to Mu-Tao Wang)} 
\begin{abstract}
In this article, we study the small sphere limit of the Wang--Yau quasi-local energy defined in \cite{Wang-Yau1,Wang-Yau2}. Given a point $p$ in a spacetime $N$, we consider a canonical family of surfaces approaching $p$ along its future null cone and evaluate the limit of the Wang--Yau quasi-local energy. The evaluation relies on solving an ``optimal embedding equation" whose solutions represent critical points of the quasi-local energy.  For a spacetime with matter fields, the scenario is similar to that of the large sphere limit found in \cite{Chen-Wang-Yau1}. Namely,  there is a natural solution which is a local minimum, and the limit of its quasi-local energy recovers the stress-energy tensor at $p$. For a vacuum spacetime, the quasi-local energy vanishes to higher order and  the solution of  the optimal embedding equation is more complicated. Nevertheless, we are able to show that there exists a solution which is a local minimum and that the limit of its quasi-local energy is related to the Bel--Robinson tensor. Together with earlier work \cite{Chen-Wang-Yau1}, this completes the consistency verification of the Wang--Yau quasi-local energy with all classical limits. 
\end{abstract}
\maketitle
\section{Introduction}
In general relativity, a spacetime is a  4-manifold $N$ with a Lorentzian metric $g_{\alpha\beta}$ satisfying the {\it Einstein equation}
\[R_{\alpha\beta}-\frac{R}{2}g_{\alpha\beta}=8\pi T_{\alpha\beta},\]
where $R_{\alpha\beta}$ and $R$ are the Ricci curvature and the scalar curvature of the metric $g_{\alpha\beta}$, respectively. 
On the right hand side of the Einstein equation, $T_{\alpha\beta}$ is the stress-energy tensor of the matter field, a divergence free and symmetric 2-tensor. For most matter fields, $T_{\alpha\beta}$ satisfies the dominant energy condition. For a vacuum spacetime where $T_{\alpha\beta}=0$ (which implies $R_{\alpha \beta} =0$), one way of measuring the 
gravitational energy  is to consider the {\it Bel--Robinson tensor} \cite{Bel} 
\[ Q_{\mu\nu \alpha \beta} = W^{\rho \,\,\,\, \sigma}_{\,\,\,\, \mu \,\,\,\, \alpha}W_{\rho \nu \sigma \beta}+W^{\rho \,\,\,\, \sigma}_{\,\,\,\, \mu \,\,\,\, \beta}W_{\rho \nu \sigma \alpha} - \frac{1}{2}g_{\mu\nu}W_{\alpha}^{\,\,\,\,  \rho \sigma \tau}W_{\beta  \rho \sigma \tau},  \]
where $W_{\alpha \beta \gamma \delta}$ is the Weyl curvature tensor of the spacetime $N$. For a vacuum spacetime, the Bel--Robinson tensor is a divergence free and totally symmetric 4-tensor which also satisfies a certain positivity condition \cite{Christodoulou-Klainerman}. The stress-energy tensor and the Bel--Robinson tensor are useful in studying the global structure of the maximal development of the initial value problem in general relativity, see for example \cite{Christodoulou-Klainerman,Bieri-Zipser}.

We recall that given a spacelike 2-surface $\Sigma$ in a spacetime $N$, the Wang--Yau quasi-local energy $E(\Sigma, X, T_0)$ is defined in \cite{Wang-Yau1,Wang-Yau2} with respect to each pair $(X, T_0)$ of an isometric embedding $X$ of $\Sigma$ into the Minkowski space $\R^{3,1}$
and a constant future timelike unit vector $T_0\in \R^{3,1}$. If the spacetime satisfies the dominant energy condition and the pair $(X,T_0)$ is admissible (see \cite[Definition 5.1]{Wang-Yau2}), it is proved that $E(\Sigma, X, T_0) \ge 0$. 
The Wang--Yau quasi-local mass is defined to be the infimum of the quasi-local energy among all admissible pairs $(X,T_0)$. The Euler--Lagrange equation for the critical points of the quasi-local energy is derived in  \cite{Wang-Yau2} and the solutions are referred to as the optimal embeddings. 
The definition of the Wang--Yau quasi-local mass and the optimal embedding equation are reviewed in more details in Section \ref{sec_review_quasi-local mass}. 

For a family of surfaces $\Sigma_r$ and a family of isometric embeddings $X_r$ of $\Sigma_r$ into $\R^{3,1}$, the limit  of $E(\Sigma_r, X_r, T_0)$ is evaluated in \cite[Theorem 2.1]{Wang-Yau3} under the compatibility condition
\begin{equation} \label{ratio}
\lim_{r \to \infty} \frac{|H_0|}{|H|}=1,
\end{equation}
where $H$ and $ H_0$ are the mean curvature vectors of $\Sigma_r$ in $N$ and  the image of $X_r$ in $\R^{3,1}$, respectively. Under the compatibility condition, the limit of $E(\Sigma_r, X_r, T_0)$ becomes a linear function of $T_0$. 

The compatibility condition \eqref{ratio} holds naturally in the following situations:
\begin{itemize}
\item A family of surfaces approaching the spatial infinity of an asymptotically flat spacetime. The limit of the Wang--Yau quasi-local energy is evaluated in \cite{Wang-Yau3}. It is proved that the 
 limit of the quasi-local energy of the coordinate spheres of an asymptotically flat initial data is the linear function dual to the Arnowitt--Deser--Misner (ADM) energy-momentum vector \cite{Arnowitt-Deser-Misner}. 
\item A family of surfaces approaching the null infinity of an asymptotically flat spacetime. The limit of the Wang--Yau quasi-local energy is evaluated in \cite{Chen-Wang-Yau1} where the null infinity is modeled using the Bondi coordinates. It is proved that the limit of the quasi-local energy of the coordinate spheres is the linear function dual to the Bondi-Trautman energy-momentum vector \cite{Bondi-Burg-Metzner,Trautman}.
\item A family of surfaces approaching a point $p$ in a spacetime along the null cone of $p$. This is the small sphere limit we study in this article. It is expected that the leading term of the quasi-local energy recovers the stress-energy tensor in spacetimes with matter fields and the Bel--Robinson tensor for vacuum spacetimes.
\end{itemize}

In this article, we confirm the last expectation and thus complete the consistency verification of the Wang--Yau energy with all classical limits. The setting for the small sphere limit is as follows:  Let $p$ be a point in a spacetime $N$.  Let $C_p$ be the future null hypersurface generated by future null geodesics  starting at $p$. Pick any future directed timelike unit vector  $e_0$ at $p$. Using $e_0$, we normalize a null vector $L$ at $p$ by 
\[  \langle L , e_0 \rangle =-1 .\]
We consider the null geodesics of the normalized $L$ and let $r$ be the affine parameter of these null geodesics. Let $\Sigma_r$ be the family of surfaces on $C_p$ defined by the level sets of the affine parameter $r$. The inward null normal $\underline L$ of $\Sigma_r$ is normalized so that 
\[  \langle L ,\underline L\rangle =-1 . \]

The Wang--Yau quasi-local mass is defined to be the infimum of the quasi-local energy among all admissible pairs $(X,T_0)$. To evaluate the small sphere limit of the quasi-local mass, we first need to understand the limit of the optimal embedding equation. For the large sphere limit, the optimal embedding equation at infinity of an asymptotically flat spacetime is solved in \cite{Chen-Wang-Yau1}. We computed the linearization of the optimal embedding equation  and obtained a unique optimal embedding which locally minimizes the energy. It is observed that the invertibility of the linearized operator comes from the positivity of the total energy. For the small sphere limit, we apply a similar approach to analyze the optimal embedding equation. For a spacetime with matter fields, there is a unique choice of the leading term of $T_0$ such that the leading term of the optimal embedding equation is solvable. Moreover, the optimal
  embedding equation can be solved by iteration.  However, for a vacuum spacetime,  the quasi-local energy vanishes to higher order and the invertibility of the optimal embedding equation is more subtle. In fact, the leading order term of the optimal embedding equation is solvable for any choice of $T_0$ and the obstruction to solve the optimal embedding equation only occurs in the third order term of the optimal embedding equation. 

The first main result of this article is the following theorem:
\begin{theorem} \label{main_theorem_1}
Let $\Sigma_r$ be the above family of surfaces approaching $p$ and with respect to $e_0$. 
\begin{enumerate}
\item For the isometric embeddings $X_r$ of $\Sigma_r$ into $\R^{3}$, the limit of the quasi-local energy $E(\Sigma_r, X_r, T_0)$ as $r$ goes to $0$ satisfies 
\[ \lim_{r \to 0} r^{-3}E(\Sigma_r, X_r,T_0) = \frac{4\pi}{3}T(e_0, T_0), \] where $T(\cdot, \cdot)$ is the stress-energy  tensor at $p$
\item Suppose $T(e_0, \cdot) $ is dual to a timelike vector $V$ at $p$. There is a family of $(X_r,T_0(r) )$  which locally minimizes the energy of $\Sigma_r$. Moreover, for this family of pairs, we have
\[ \lim_{r \to 0} r^{-3}E(\Sigma_r, X_r,T_0(r) ) = \frac{4\pi}{3} \sqrt{-\langle V, V\rangle}. \] 
\end{enumerate}
\end{theorem}

For a vacuum spacetime, the quasi-local energy vanishes to higher order. We proved that the leading order term of the Wang--Yau quasi-local energy is of $O(r^5)$. In order to describe the result in the vacuum case, we pick a local coordinate system $(x^0, x^1, x^2, x^3)$ near $p$ such that $e_0=\frac{\partial}{\partial x^0}$ at $p$ and denote by $\bar{W}_{0kmn}$ the value of $W(e_0, \frac{\partial}{\partial x^k},\frac{\partial}{\partial x^m},\frac{\partial}{\partial x^n})$ at $p$
and by $\bar{W}_{0m0n}$ the value of $W(e_0, \frac{\partial}{\partial x^m}, e_0,\frac{\partial}{\partial x^n})$ at $p$, etc. From the definition of the Bel--Robinson tensor, for $T_0=(a^0, a^1, a^2, a^3)$, we have
\[  Q(e_0,e_0,e_0,T_0)= (\frac{1}{2}\sum_{k,m,n} \bar W_{0kmn}^2+\sum_{m,n} \bar W_{0m0n}^2 )a^0  + 2\sum_{m,n,i} \bar W_{0m0n} \bar W_{0min} a^i . \]

We prove the following theorem:
\begin{theorem}
Let $\Sigma_r$ be the above family of surfaces approaching $p$ and with respect to $e_0$. Suppose the stress-energy tensor vanishes in an open set in $N$ containing $p$.  
\begin{enumerate}
\item For each observer $T_0$, there is a pair $(X_r(T_0),T_0)$ solving the leading order term of the optimal embedding equation of $\Sigma_r$ (see Lemma \ref{xi3} and \ref{x03}). For this pair $(X_r(T_0),T_0)$, we have 
\[  
\begin{split}
     & \lim_{r \to 0} r^{-5} E(\Sigma_r, X_r(T_0),T_0)= \frac{1}{90}  \Big[  Q(e_0,e_0,e_0,T_0)+ \frac{\sum_{m,n} \bar W_{0m0n}^2}{2a^0}\Big].
\end{split}
\]

\item Suppose $Q(e_0,e_0,e_0,\cdot)$ is dual to a timelike vector. Let $\mathcal P$ denote the set of $(X,T_0)$ admitting a power series expansion given in equation \eqref{assume}. We have 
\[  
\begin{split}
    &\inf_{(X,T_0) \in \mathcal P} \lim_{r \to 0} r^{-5} E(\Sigma_r, X,T_0) 
 =  \inf_{(a^0,a^i) \in \mathbb H^3} \frac{1}{90}  \Big[ Q(e_0,e_0,e_0,T_0) + \frac{\sum_{m,n} \bar W_{0m0n}^2}{2a^0}\Big].
\end{split}
\]
where $\mathbb H^3$ denotes the set of unit timelike future directed vector in $\R^{3,1}$. The infimum is achieved by a unique $(a^0,a^i) \in \mathbb H^3$.
\end{enumerate}
\end{theorem}

There have been many previous work \cite{blk2, fst, hs, yu} on evaluating the small sphere limit of different notions of quasi-local energy. In \cite{blk2}, Brown, York, and Lau evaluated the small sphere limit of the canonical QLE (a modified Brown--York energy). 
In \cite{yu}, Yu evaluated the small sphere limit of  a modified Liu--Yau mass (with light cone reference embedding instead of $\R^3$ reference). 
In both of the above papers, the family of physical surfaces considered is the same as the one we use in this article. In \cite{fst}, Fan,  Shi, and Tam evaluated the small sphere limit of the Brown--York mass and the Hawking mass, for surfaces approaching a point in a totally geodesic hypersurface in a spacetime. See \cite{sz} for a survey of different notions of quasi-local energy and their limiting behaviors. 

When the stress-energy density $T_{\alpha \beta}$ is non-vanishing at the point $p$, all the above small sphere limits of different notions of quasi-local energy give the same result. Namely, the leading order term is $\frac{4 \pi r^3}{3}T(e_0,e_0)$. However, if the stress energy density vanishes near $p$, then the above results give different answers, although all of them are of order $O(r^5)$ and are related to the Bel--Robinson tensor. We compare our calculation with others in the following:

1) In our case,  the Lorentzian symmetry is recovered and our energy and momentum components of the limits are computed from the same quasi-local energy expression. By comparison, in the previous results,  either there is only the energy component or the momentum components come from a separate definition.

2) In our limit in a vacuum spacetime, the isometric embedding equation into $\R^3$ is explicitly solved and used to evaluate the limit. While the existence of such isometric embedding is guaranteed by the solution of the Weyl problem by Nirenberg \cite{n} and Pogorelov \cite{Pogorelov}, computing the embedding in terms of the Weyl curvatures at $p$ is needed for the evaluation of the small sphere limit in a vacuum spacetime. 
This difficulty is circumvented in several previous works by using the light cone reference instead.

In Section \ref{sec_physicaldata}, we compute the expansion of the induced metric, the second fundamental forms and the  connection 1-form of the family of surfaces $\Sigma_r$. Using this information, we compute the non-vacuum small sphere limit of the Wang--Yau quasi-local energy in Section \ref{sec_limit_non}. The rest of the paper is devoted to the small sphere limit in vacuum spacetimes. In Section \ref{sec_weyl_at}, we compute several functions, tensors, and integrals on $S^2$ that will be used repeatedly in later sections.  In section \ref{sec_optimal_va}, for each observer $T_0$, we  compute the $O(r^3)$ term of the isometric embedding using the leading order term of the optimal embedding equation. The  isometric embeddings, which depends on the choice of $T_0$, is denoted by $X_r(T_0)$. In the next four sections, the quasi-local energy associated to the pair $(X_r(T_0),T_0)$ is computed. Section \ref{sec_energy_com_va}, \ref{sec_ref_ham} and \ref{sec_phy_ham} are used to compute
  the three separate terms in the quasi-local energy and these results are combined in Section \ref{sec_eva_energy} to evaluate $E(\Sigma_r,X_r(T_0),T_0)$.  In Section \ref{sec_energy_min}, we show that there is exactly one $T_0$ which minimizes $E(\Sigma_r,X_r(T_0),T_0)$ and study the optimal embedding equation in further details.

\section{Review of the Wang--Yau quasi-local mass} \label{sec_review_quasi-local mass}
Let $\Sigma$ be a closed embedded spacelike 2-surface in the spacetime $N$. We assume the mean curvature vector $H$ of $\Sigma$ is spacelike.  Let $J$ be the reflection of $H$ through the future outgoing  light cone in the normal bundle of $\Sigma$. 
The data used in the definition of the Wang--Yau quasi-local energy is the triple $(\sigma,|H|,\alpha_H)$ on $\Sigma$ where $\sigma$ is the induced metric, $|H|$ is the norm of the mean curvature vector, and $\alpha_H$ is the connection 1-form of the normal bundle with respect to the mean curvature vector
\[ \alpha_H(\cdot )=\langle \nabla^N_{(\cdot)}   \frac{J}{|H|}, \frac{H}{|H|}   \rangle  \]
where $\nabla^N$ is the covariant derivative in $N$.

Given an isometric embedding $X:\Sigma\rightarrow \R^{3,1}$ and a future timelike unit vector $T_0$ in $\R^{3,1}$, suppose the projection $\widehat{X}$ of $X(\Sigma)$ onto the orthogonal complement of $T_0$ is embedded, and denote the induced metric, 
the second fundamental form, and the mean curvature of the image surface $\widehat{\Sigma}$ of $\widehat{X}$  by $\hat{\sigma}_{ab}$, $\hat{h}_{ab}$, and $\widehat{H}$, respectively.  The Wang--Yau quasi-local energy $E(\Sigma, X, T_0)$  of $\Sigma$ with respect to the pair $(X, T_0)$ is 
\[
\begin{split}E(\Sigma, X, T_0) 
=  \frac{1}{8\pi}\int_{\widehat{\Sigma}} \widehat{H} d{\widehat{\Sigma}}+ \frac{1}{8\pi}\int_\Sigma \left[\sqrt{1+|\nabla\tau|^2}\cosh\theta|{H}| + \nabla  \tau  \cdot \nabla \theta  -\alpha_H ( \nabla \tau) \right]d\Sigma,\end{split}\] 
where 
\[\theta=\sinh^{-1}(\frac{-\Delta\tau}{|H|\sqrt{1+|\nabla\tau|^2}}),\]
$\nabla$ and $\Delta$ are the gradient and Laplace operator of $\sigma$, respectively, and $\tau=-\langle X, T_0\rangle$ is considered as the time function on $\Sigma$.

If the spacetime satisfies the dominant energy condition, $\Sigma$ bounds a spacelike hypersurface in $N$, and the pair $(X,T_0)$ is admissible,  it is proved in \cite{Wang-Yau2} that $E(\Sigma, X, T_0)  \ge 0$.  The Wang--Yau quasi-local mass is defined to be the infimum of the quasi-local energy $E(\Sigma, X, T_0)$ among all admissible pairs $(X,T_0)$.  
The Euler--Lagrange equation for a critical point $(X_0, T_0)$ of the quasi-local energy is the elliptic equation
\[   -(\widehat{H}\hat{\sigma}^{ab} -\hat{\sigma}^{ac} \hat{\sigma}^{bd} \hat{h}_{cd})\frac{\nabla_b\nabla_a \tau}{\sqrt{1+|\nabla\tau|^2}}+ div_\sigma (\frac{\nabla\tau}{\sqrt{1+|\nabla\tau|^2}} \cosh\theta|{H}|-\nabla\theta-\alpha_{H})=0
\]  
coupled with the isometric embedding equation for $X$. This Euler--Lagrange equation is referred to as the optimal embedding equation and a solution is referred to as an optimal embedding.

The corresponding data for the image of the isometric embedding in the Minkowski space can be used to simplify the expression for the quasi-local energy and the optimal embedding equation. Denote the norm of the mean curvature vector and the connection 1-form with respect to the mean curvature vector of $X(\Sigma)$ in $\R^{3,1}$ by $|H_0|$ and $\alpha_{H_0}$, respectively. We have the following identities relating the geometry of the image of the isometric embedding $X$ and the image surface $\widehat{\Sigma}$ of $\widehat{X}$.
\[\sqrt{1+|\nabla\tau|^2}\widehat{H} =\sqrt{1+|\nabla\tau|^2}\cosh\theta_0|{H_0}|-\nabla\tau\cdot \nabla \theta_0 -\alpha_{H_0} ( \nabla \tau), \]
\[   -(\widehat{H}\hat{\sigma}^{ab} -\hat{\sigma}^{ac} \hat{\sigma}^{bd} \hat{h}_{cd})\frac{\nabla_b\nabla_a \tau}{\sqrt{1+|\nabla\tau|^2}}+ div_\sigma (\frac{\nabla\tau}{\sqrt{1+|\nabla\tau|^2}} \cosh\theta_0|{H_0}|-\nabla\theta_0-\alpha_{H_0})=0,
\]
where
\[\theta_0=\sinh^{-1}(\frac{-\Delta\tau}{|H_0|\sqrt{1+|\nabla\tau|^2}}).\]
The first identity is derived in \cite[Proposition 3.1]{Wang-Yau2}. The second identity simply states that a surface inside $\R^{3,1}$ is a critical point of the quasi-local energy with respect to isometric embeddings of the surface back to $\R^{3,1}$ with other time functions. This can be proved by either the positivity of the Wang--Yau quasi-local energy or by a direct computation \cite{Chen-Wang-Yau2}.

We substitute these relations into the expression for $E(\Sigma, X, T_0)$ and the optimal embedding equation and rewrite them in
term of the following function $f$ (In \cite{Chen-Wang-Yau3}, we denote this function by $\rho$.):
\begin{equation} \label{density}
f= \frac{\sqrt{|H_0|^2 +\frac{(\Delta \tau)^2}{1+ |\nabla \tau|^2}} - \sqrt{|H|^2 +\frac{(\Delta \tau)^2}{1+ |\nabla \tau|^2}} }{ \sqrt{1+ |\nabla \tau|^2}} .
\end{equation}

The quasi-local energy becomes
\begin{equation}\label{qle} E(\Sigma, X, T_0)=\frac{1}{8\pi}\int_\Sigma \left[f (1+|\nabla\tau|^2)+\Delta \tau \sinh^{-1}(\frac{f \Delta\tau}{|H_0| |H|})-\alpha_{H_0}(\nabla \tau)+\alpha_H(\nabla \tau)\right] d\Sigma, \end{equation}
and the optimal embedding equation becomes
\begin{equation} \label{optimal3}
div_\sigma\left(f \nabla \tau - \nabla [ \sinh^{-1} (\frac{f \Delta \tau }{|H_0||H|})] - \alpha_{H_0} + \alpha_{H}\right)=0.
\end{equation}
In \cite{Chen-Wang-Yau1}, we studied the large sphere limit of the optimal embedding equation using the expression in equation \eqref{optimal3} and derived an iteration scheme to find a solution which minimizes the energy. In later sections, we analyze the small limit of the optimal embedding equation using the same expression. 
\section{The Expansion of the physical data}\label{sec_physicaldata}
Let $\Sigma_r$ be the family of surfaces approaching a point $p$ in the spacetime $N$ constructed in the introduction. In this section, we compute  the expansions of the induced metric, the second fundamental forms and the  connection 1-form of $\Sigma_r$. The result here is essentially the same as the expansions computed in \cite{blk2} using the Newman-Penrose formalism. However, the quantities are computed in a different frame to adapt to our situation.  This simplifies the computation in the later sections.

\subsection{Leading order expansion in non-vacuum spacetimes.}
In this subsection, we compute the expansion of the geometric quantities in terms of the affine parameter $r$. We compute enough terms in the expansion in order to evaluate the small sphere limit of the Wang--Yau quasi-local mass in spacetimes with matter fields. Expansions to higher order in vacuum spacetimes are given in the next subsection. 

We parametrize $\Sigma_r$ in the following way. Consider a smooth map 
\begin{equation}\label{parametrization} X:S^2\times [0, \epsilon)\rightarrow N\end{equation} such that
for each fixed point in $S^2$, $X(\cdot, r), r\in [0, \epsilon)$ is a null geodesic parametrized by the affine parameter $r$, with $X(\cdot, 0)=p$
and $\frac{\partial X}{\partial r}(\cdot, 0)\in T_pN$ a null vector such that $\langle \frac{\partial X}{\partial r}(\cdot, 0), e_0\rangle=-1$ .
Let  $L=\frac{\partial X}{\partial r}$ be the null generator, $\nabla^N_L L=0$. We also choose a local coordinate system $\{u^a\}_{a=1, 2}$ on $S^2$ such that $\partial_a=\frac{\partial X}{\partial u^a}, a=1,2$ form a tangent basis to $\Sigma_r$. Let $\underline L$ be the null normal vector field along $\Sigma_r$ such that $\langle L, \underline{L}\rangle=-1$.  
Denote
\[
\begin{split}
 l_{ab} = & \langle   \nabla^N_{\partial_a}  \partial_b , L\rangle    \\
 n_{ab} =& \langle   \nabla^N_{\partial_a}  \partial_b , \underline L\rangle   \\
 \eta_a  = &  \langle   \nabla^N_{L} \partial_a , \underline L \rangle   
\end{split}
\]
for the second fundamental forms in the direction of $L$ and $\underline L$ and the connection 1-form in the null normal frame, respectively. We consider these as tensors on $S^2$ that depend on $r$ and use the induced metric on $\Sigma_r$, $\sigma_{ab}=\langle \partial_a, \partial_b\rangle$ to raise or lower indexes. Let $\nabla$ and $\Delta$ be the covariant derivative and the Laplacian with respect to $\sigma$, respectively.

In particular, we have 
\begin{equation}\label{tangent_covariant}
\begin{split}
 \nabla^N_{\partial_a} L&=-l_a^c \partial_c-\eta_a L\\
 \nabla^N_{\partial_a} \partial_b&=\gamma_{ab}^c\partial_c-l_{ab}\underline L-n_{ab}L\\
 \nabla^N_{\partial_a}\underline L&=-n_a^c\partial_c+\eta_a \underline L,
\end{split}
\end{equation} where $\gamma_{ab}^c$ are the Christoffel symbols of $\sigma_{ab}$. Let
\[
\begin{split}
\hat{l}_{ab}= &l_{ab}-\frac{1}{2}(\sigma^{cd}l_{cd})\sigma_{ab} \\
\hat{n}_{ab}=& l_{ab}-\frac{1}{2}(\sigma^{cd}l_{cd})\sigma_{ab}
\end{split}
\]
be the traceless part of $l_{ab}$ and $n_{ab}$. 

The following identities for covariant derivatives are useful.
\[
\begin{split} 
\nabla^N_{L}  \partial_a  =&  -l_a^c \partial_c - \eta_a L  \\
 \nabla^N_{L}  \underline L  = &  -\eta^b \partial_b .
\end{split}
\]

We first derive the following:
\begin{lemma}
The induced metric, the second fundamental forms and the connection 1-form satisfy the following differential equations:
\begin{equation}\label{eq_sigma}  \partial_r \sigma_{ab} = -2 l_{ab} \end{equation}
\begin{equation}\label{eq_l}\partial_r l_{ab} = R_{LabL} -l_{ac} l^c_b\end{equation}
\begin{equation}\label{eq_n} \partial_r n_{ab} = R_{Lab\underline L}  - l_b^c n_{ac} +  \nabla_a  \eta_b - \eta_a\eta_b\end{equation}
\begin{equation}\label{eq_eta}\partial_r \eta_a = R_{LaL\underline L} +l_a^b \eta_b\end{equation}
\begin{equation} \label{eq_ray}\partial_r (\sigma^{ab} l_{ab})= \frac{1}{2}(\sigma^{ab} l_{ab}) ^2 + \hat{l}_a^b \hat{l}_b^a + Ric(L,L)\end{equation}
\begin{equation}\label{eq_trn}\partial_r (\sigma^{ab} n_{ab}) =  Ric(L,\underline L)+R_{L\underline LL\underline L}+  l^{ab} n_{ab} + div_{\sigma} \eta - \eta_a \eta^a. \end{equation}
$Ric$ and $R_{\alpha \beta \gamma \delta}$ are
 the Ricci curvature and the full Riemannian curvature tensor of the spacetime N, respectively. 
\end{lemma}

\begin{proof}
Equation \eqref{eq_sigma}  follows from the definition of second fundamental form. For $l_{ab}$,
\begin{align*}
\partial_r l_{ab} ={}& \partial_r\langle \nabla^N_{\partial_a} \partial_b, L \rangle \\
                              ={}& R_{LabL} + \langle \nabla^N_{\partial_a} \nabla^N_{L}\partial_b, L \rangle \\
                              ={}& R_{LabL} -l_{ac} l^c_b.
\end{align*}

We derive \[\partial_r(\sigma^{ab} l_{ab})=l^{ab} l_{ab}+ Ric(L,L).\] 
Equation \eqref{eq_ray} follows from the decomposition $l^{ab}l_{ab}=\hat{l}^{ab} \hat{l}_{ab}+\frac{1}{2}(\sigma^{ab}l_{ab})^2$. Indeed, this is the Raychaudhuri equation.

For the connection 1-form, we have
\[
\begin{split}
\partial_r \eta_a = & \partial_r \langle \nabla^N_{\partial_a}L, \underline L \rangle \\
=& R_{LaL \underline L} +l_a^b \eta_b.
\end{split}
\]
For $n_{ab}$, we have
\begin{equation} \label{nab}
\begin{split}  \partial_r n_{ab} ={} & \partial _r \langle \nabla^N_{\partial_a} \partial_b,\underline L \rangle \\
={} & R_{Lab\underline L } + \langle \nabla^N_{\partial_a} (\nabla^N_{L} \partial_b) , \underline L \rangle + \langle \nabla^N_{\partial_a} \partial_b, \nabla^N_{L}\underline L \rangle \\
={} & R_{Lab\underline L } - \langle \nabla^N_{\partial_a} (l_b^c\partial_c+\eta_b L) , \underline L \rangle - \langle \nabla^N_{\partial_a} \partial_b,  \eta^c \partial_c \rangle \\
={} & R_{Lab\underline L}  - l_b^c n_{ac} +  \nabla_a  \eta_b - \eta_a\eta_b
\end{split}
\end{equation}
and
\begin{equation}
\begin{split}
\partial_r (\sigma^{ab} n_{ab}) = & \sigma^{ab} \partial_r n_{ab} + ( \partial_r \sigma^{ab})n_{ab}  \\
                      = & \sigma^{ab}R_{Lab\underline L}  - l^{ab} n_{ab} + div_{\sigma} \eta - \eta^a\eta_a  + 2 l^{ab} n_{ab} \\
                      = & Ric(L,\underline L)+R_{L\underline LL\underline L}+  l^{ab} n_{ab} + div_{\sigma} \eta - \eta^a\eta_a.
                      \end{split}
\end{equation}

\end{proof}

In the rest of this subsection, we consider the expansions of geometric data as $r\rightarrow 0$. 
As we remarked, we consider $\sigma_{ab}, l_{ab}, n_{ab}, \eta_a$ as tensors on $S^2\times [0, \epsilon)$, or 
tensors on $S^2$ that depend on the parameter $r$.  We shall see below that they have the following expansions. 
\[\begin{split}
\sigma_{ab}&=\tilde{\sigma}_{ab} r^2+O(r^3)\\
l_{ab}&=-\tilde{\sigma}_{ab} r+O(r^2)\\
n_{ab}&=\frac{1}{2} \tilde{\sigma}_{ab} r+O(r^2)\\
\eta_{a}&=\frac{1}{3} {\beta}_a r^2+O(r^3)
\end{split}\] where $\beta_a=\lim_{r\rightarrow 0} R_{La L\underline L}$ is considered as a $(0,1)$ tensor on $S^2$, $\tilde \sigma_{ab}$ denotes the standard metric on unit $S^2$. Let $\tilde \nabla$ and $\tilde \Delta$ be the covariant derivative and the Laplacian with respect to $\tilde \sigma_{ab}$, respectively. 

We shall also consider the pull-back of tensors from the null hypersurface. For example, we consider 
$R(L, \cdot, L, \underline L)$ as a tensor defined on $C_p$ and take its pull-back through \eqref{parametrization}, which is then consider as a 
$(0,1)$ tenors on $S^2$ that depends on $r$ (or on $S^2\times [0, \epsilon)$). We shall abuse the notations and still denote the pull-back tensor by $R_{L a L\underline L}$. 
In particular, $R_{LabL }, R_{LaL\underline L}, R_{L\underline LL\underline L  }$ are considered as $r$ dependent $(0, 2)$ tensor,
$(0, 1)$ tensor, and a scalar function on $S^2$, respectively, of the following orders 
\[ R_{LabL } = O( r^2), \,\, R_{LaL\underline L}= O(r) \,\,  {\rm and} \,\, R_{L\underline LL\underline L  } = O(1). \]

To describe their expansions, we first write down the expansions of $L$ and $\partial_a$. Let $x^0, x^i, i=1, 2, 3$ be a normal coordinates system at $p$ such that the original future timelike vector $e_0\in T_p N$ is $\frac{\partial}{\partial x^0}$. The parametrization \eqref{parametrization} is given by 
\[X(u^a, r)=X^0(u^a, r)\frac{\partial}{\partial x^0}+X^i(u^a, r)\frac{\partial}{\partial x^i}\] with the following expansions:
\[\begin{split} X^0(u^a, r)&=r+O(r^2)\\
X^i(u^a, r)&=r \tilde{X}^i (u^a)+O(r^2), \end{split}\] where $\tilde{X}^i(u^a)$ are the three first eigenfunctions of the standard metric $\tilde{\sigma}_{ab}$ on $S^2$. For example, if we take the coordinates $u_a, a=1,2$ to be the standard spherical coordinate system $\theta, \phi$ with $\tilde{\sigma}=d\theta^2+\sin^2\theta d\phi^2$, then $\tilde{X}^1=\sin\theta\sin\phi$, $\tilde{X}^2=\sin\theta \cos\phi$, and $\tilde{X}^3=\cos\theta$.
In particular,  
\begin{equation}\label{frame}\begin{split} L&=\frac{\partial X}{\partial r}=\frac{\partial}{\partial x^0}+\tilde{X}^i(u^a) \frac{\partial}{\partial x^i}+O(r)\\ 
\partial_a&=\frac{\partial X}{\partial u^a}=r\frac{\partial \tilde{X}^i}{\partial u^a}\frac{\partial}{\partial x^i}+O(r^2),\,\, a=1, 2.\end{split}\end{equation}

We have the following expansions for the curvature tensor:
\begin{equation}
\begin{split}\label{exp_R_LaLN}
 R_{LabL } =  & r^2 \bar R_{LabL} +O(r^3)  \\
R_{LaL\underline L}=& r  \bar R_{L a L\underline L} +O(r^2) \\
 R_{L\underline LL\underline L  } =&  \bar R_{L\underline LL\underline L}+O(r), 
\end{split}
\end{equation}
where  $\bar R_{LabL}$, $\bar R_{LaL\underline L}$ and $\bar R_{L\underline LL\underline L  }$ correspond to the appropriate rescaled limit of the respective
tensors as $r\rightarrow 0$. For example, 
\[\bar R_{LabL}=\lim_{r\rightarrow 0} \frac{1}{r^2} R_{LabL}=R(\frac{\partial}{\partial x^0}+\tilde{X}^i \frac{\partial}{\partial x^i},   \frac{\partial \tilde{X}^j}{\partial u^a},\frac{\partial \tilde{X}^k}{\partial u^b}, \frac{\partial}{\partial x^0}+\tilde{X}^l\frac{\partial}{\partial x^l})(p)\] is considered as a $(0, 2)$ tensor on the standard $S^2$. 

\begin{lemma} \label{expansion_first}
We have the following expansions:
\begin{align}\label{l} l_{ab}= & - r \tilde \sigma_{ab} +\frac{2}{3}r^3 \bar R_{La b L}+ O(r^4) \\
  \label{sigma}
\sigma_{ab} =& r^2\tilde\sigma_{ab} - \frac{1}{3} r^4\bar R_{La b L} + O(r^5)\\
\label{l_contracted}l_a^c= &-r^{-1} \delta_a^c+\frac{1}{3} r \bar R_{La b L}\tilde{\sigma}^{bc}+ O(r^2)\\
 \label{eta_first_order}\eta_a = &\frac{1}{3}r^2 \bar R_{La L\underline L}+O(r^3).
\end{align}
\end{lemma}

\begin{proof}
Suppose we have the expansions
\[
\begin{split}
\sigma_{ab}=& r^2\tilde{\sigma}_{ab}+r^3\sigma^{(3)}_{ab}+r^4\sigma^{(4)}_{ab}+O(r^5)\\
\eta_a = &r^2 \eta_a^{(2)} +O(r^3).
\end{split}
\] 
Using  $R_{LabL } =  r^2 \bar R_{LabL} +O(r^3)$ and equation \eqref{eq_l}, we conclude
\begin{equation} l_{ab}= - r \tilde \sigma_{ab} +\frac{2}{3}r^3\bar R_{Lab L}+ O(r^4). \end{equation}
From  equation \eqref{eq_sigma}, this implies
\begin{equation} 
\sigma_{ab} = r^2\tilde\sigma_{ab} - \frac{r^4}{3}\bar R_{La b L} + O(r^5).
\end{equation}
From equation  \eqref{eq_eta} for $\eta$ and the expansion for $l_{ab}$, we derive
\[  \eta_a = \frac{1}{3}r^2\bar R_{La L\underline L}+O(r^3). \]
\end{proof}
Next, we derive the expansion for the mean curvature in the direction of $L$ and $\underline L$.
\begin{lemma} \label{data_non}
We have the following expansions for  $\sigma^{ab} l_{ab}$ and $\sigma^{ab} n_{ab}$
\begin{align}
\sigma^{ab} l_{ab} =& - \frac{2}{r} +\frac{  1}{3} r \bar Ric(L,L)    +O(r^2), \\
\sigma^{ab} n_{ab} =& \frac{1}{r} + r(\sigma^{ab} n_{ab})^{(1)}+O(r^2),
\end{align}
where
\[(\sigma^{ab} n_{ab})^{(1)}=\frac{1}{2}\bar R_{L\underline L L\underline L}+\frac{1}{2}\bar  Ric(L,\underline L)+\frac{1}{12} \bar Ric(L,L)+\frac{1}{6} \tilde \nabla ^a \bar  R_{L a L\underline L}.\]
\end{lemma}
\begin{proof}
The expansion for $\sigma^{ab} l_{ab}$ follows from the expansions for $l_{ab}$ and $\sigma_{ab}$. 
Equation \eqref{eq_n}  for $n_{ab}$ is equivalent to
\[  r\partial_r (r^{-1} n_{ab}) = R_{Lab\underline L}  - (l_b^c+r^{-1}\delta_b^c) n_{ac} +  \nabla_a  \eta_b - \eta_a\eta_b.\]
Since $n_{ab} = \frac{1}{2} r \tilde \sigma_{ab} +O(r^2)$, it follows that
\[  r\partial_r (r^{-1} n_{ab}) =r^2( \bar R_{Lab\underline L}  -\frac{1}{6} \bar R_{LabL} +   \tilde \nabla _a  \eta^{(2)}_b)+O(r^3).\]
Integrating with respect to $r$, we obtain
\[ n_{ab} =\frac{1}{2} r \tilde{\sigma}_{ab} + \frac{1}{2} r^3  (  \tilde \nabla_a \eta_{b}^{(2)} +  \bar R_{L ab \underline L}-\frac{1}{6}\bar R_{LabL}) + O(r^4) . \]
Using the expansion for $\sigma_{ab}$ again, we finish the proof the the lemma.
\end{proof}

The term $(\sigma^{ab} n_{ab})^{(1)}$ can by further simplified by the following lemma.
\begin{lemma} \label{div_non}
\[\tilde \nabla^b (\bar{R}_{LbL\underline L})=3\bar{R}_{L\underline LL\underline L}+\bar Ric(L,\underline L)+ \frac{1}{2}\bar Ric(L,L).\]
\end{lemma}
\begin{proof}
For this, we go back to general $\Sigma_r$ and compute 
\[\sigma^{ab}(\partial_a(R_{LbL\underline L})-\gamma_{ab}^cR_{LcL\underline L})\] on $\Sigma_r$. This term is of order $\frac{1}{r}$ and the leading coefficient is $\tilde{\sigma}^{ab} \tilde \nabla _a (\bar{R}_{LbL\underline L})$. On the other hand, we use the Leibniz rule and derive \[\begin{split}&\sigma^{ab}(\partial_a(R_{LbL\underline L})-\gamma_{ab}^cR_{LcL\underline L})\\
=& \sigma^{ab}(\nabla^N_aR)_{LbL\underline L}-(\sigma^{ab} l_{ab})R_{L\underline LL\underline L}
-l^{bc}R_{Lbc\underline L}-\sigma^{ab}\eta_a R_{LbL\underline L}+n^{bc}R_{LbcL}.\end{split}\]
Leading terms of the $O(\frac{1}{r})$ are the second, the third, and the fifth terms and the coefficient is $3\bar{R}_{L\underline LL\underline L}+\bar Ric(L,\underline L)$.
\end{proof}

In summary, we have the following expansions on the surfaces $\Sigma_r$:
\begin{lemma} \label{non_physical_data}
We have the following expansions for the data $(\sigma, |H| , div \alpha_H)$ on $S^2$:
\begin{equation}
\begin{split}
\sigma_{ab} = &  r^2\tilde\sigma_{ab} - \frac{1}{3} r^4\bar R_{L a b L} + O(r^5) \\
|H|^2 = & \frac{4}{r^2} + [2 \bar R_{L\underline LL\underline L} + \frac{4}{3}\bar Ric(L,\underline L) + \frac{1}{3} \bar Ric(L,L)] + O(r)  \\
 div_\sigma   \alpha_H = &  \tilde\Delta \left [  \frac{1}{2} \bar R_{L\underline LL\underline L} + \frac{1}{6} \bar Ric(L,L)  + \frac{1}{3}\bar Ric(L,\underline L) \right ] \\ 
&  -\bar R_{L\underline LL\underline L} - \frac{1}{3}\bar Ric(L,\underline L)- \frac{1}{6}\bar Ric(L,L) + O(r).
\end{split}
\end{equation}
\end{lemma}
\begin{proof}
This follows from
\begin{equation}
\begin{split}
|H|^2 = & -2 (\sigma^{ab}l_{ab})\cdot (\sigma^{cd} n_{cd})\\
 div   \alpha_H = &  -\frac{1}{2} \Delta \ln ( -\sigma^{ab}l_{ab}) + \frac{1}{2} \Delta \ln ( \sigma^{ab} n_{ab}) - div_{\sigma} \eta,
\end{split}
\end{equation}
the expansions of $\sigma$ and $\eta$ in Lemma \ref{expansion_first} and the expansions of $ \sigma^{ab}l_{ab}$ and  $ \sigma^{ab} n_{ab}$  in Lemma \ref{data_non}. We also apply Lemma \ref{div_non} to compute the divergence.
\end{proof}
\subsection{Further expansions in vacuum spacetimes}
In this subsection, we assume the spacetime is vacuum and compute the higher order terms in the expansions for the physical data. Again, enough expansions are obtained to evaluate the leading term of the small sphere limit of the Wang--Yau quasi-local mass in vacuum. In a vacuum spacetime, the only non-vanishing components of the curvature tensor are the Weyl curvature tensor. We decompose the Weyl curvature tensor at the point $p$ using the null frame $\{ e_a, L ,  \underline L\}$ following the notation of Christodoulou and Klainerman in \cite{Christodoulou-Klainerman} (our convention is $\langle L, \underline{L}\rangle=-1$ though):

\begin{equation}\label{Weyl} \begin{split}\alpha_{ab}&=\bar W_{aLbL}\\
 \underline{\alpha}_{ab}&=\bar W_{a\underline{L} b\underline{L}}\\
 \beta_a&=\bar W_{a L\underline{L}L}\\ 
\underline{\beta}_a&=\bar W_{a\underline{L}\underline{L} L}\\
 \rho&=\bar W_{\underline{L}L\underline{L}L}\\
\sigma&=\epsilon^{ab}\bar W_{ab\underline{L}L}.\end{split}\end{equation}

From the vacuum condition and the Bianchi equations, we obtain the following relations:
\begin{equation}\label{relations}\begin{split}
\bar W_{L a b \underline L} &= \frac{1}{2} \tilde \sigma_{ab} \rho + \frac{1}{4}\epsilon_{ab}\sigma\\
\bar W_{abcL}&=-\epsilon_{ab} \epsilon_{cd} \beta^d\\
\bar W_{abc\underline L}&=\epsilon_{ab} \epsilon_{cd} \underline{\beta}^d\\
\bar W_{ab\underline{L} L}&=\frac{1}{2} \epsilon_{ab}\sigma.\end{split}\end{equation}

All $\alpha, \underline\alpha, \beta, \underline\beta, \rho$ and $\sigma$ are considered as tensors on $S^2$ through the limiting process described above and we compute the covariant derivatives of them with respect to the standard metric $\tilde{\sigma}_{ab}$.

\begin{lemma}
\begin{equation}\label{Weyl_derivatives1}
\begin{split}
\tilde{\nabla}_c \alpha_{ab}= &(\tilde{\sigma}_{ca}\tilde{\sigma}_{bd}+\tilde{\sigma}_{cb}\tilde{\sigma}_{ad}+\epsilon_{ca}
\epsilon_{bd}+\epsilon_{cb}
\epsilon_{ad})\beta^d\\
\tilde{\nabla}_c \underline{\alpha}_{ab}= &\frac{1}{2} (\tilde{\sigma}_{ca}\tilde{\sigma}_{bd}+\tilde{\sigma}_{cb}\tilde{\sigma}_{ad}+\epsilon_{ca}
\epsilon_{bd}+\epsilon_{cb}
\epsilon_{ad})\underline{\beta}^d\\
\tilde \nabla_a \beta_b = &-\frac{3}{4} \sigma \epsilon_{ab}  + \frac{3}{2} \rho \tilde \sigma_{ab}- \frac{1}{2} \alpha_{ab}\\
\tilde \nabla_a \underline{\beta}_b = &\frac{3}{8} \sigma \epsilon_{ab}  + \frac{3}{4} \rho \tilde \sigma_{ab}- \underline{\alpha}_{ab}\\
\tilde{\nabla}_a \rho= &-\beta_a-2\underline{\beta}_a\\
\tilde{\nabla}_a \sigma= &2\epsilon_{ab}(\beta^b-2\underline{\beta}^b).
\end{split}
 \end{equation}
\end{lemma}

\begin{proof}

All of them are defined as limiting quantities as $r\rightarrow 0$ and we can represent them in terms of the limiting frame as $r \rightarrow 0$:
\begin{equation}\begin{split}\label{limit_frame} L&=\frac{\partial}{\partial x^0}+\tilde{X}^i \frac{\partial}{\partial x^i}\\
\underline L&=\frac{1}{2}(\frac{\partial}{\partial x^0}-\tilde{X}^i\frac{\partial}{\partial x^i})\\
e_a&=(\tilde{\nabla}_a \tilde{X}^i) \frac{\partial}{\partial x^i}.\end{split}\end{equation}

To compute the covariant derivative of $\alpha_{ab}$, we expand 
\[\alpha_{ab}=\tilde{\nabla}_a \tilde{X}^i \tilde{\nabla}_b \tilde{X}^k [\bar W_{i0k0}+ \tilde{X}^l \bar W_{i0kl}+\tilde{X}^j \bar W_{ijk0}
+\tilde{X}^j \tilde{X}^l \bar W_{ijkl}]\] and differentiate using the relations:
\[\tilde{\nabla}_a \tilde{\nabla}_a \tilde{X}^i=-\tilde{\sigma}_{ab} \tilde{X}^i,  \sum_{i=1}^3\tilde{X}^i \tilde{\nabla}_a \tilde{X}^i=0, \sum_{i=1}^3\tilde{\nabla}_a \tilde{X}^i \tilde{\nabla}_b \tilde{X}^i
=\tilde{\sigma}_{ab}, \tilde{\nabla}^a \tilde{X}^i \tilde{\nabla}_a \tilde{X}^j=\delta_{ij}-\tilde{X}^i \tilde{X}^j\] where $\tilde{\sigma}_{ab}$ is used to raise or lower indexes. 
Note that all the Weyl curvature coefficients are constants valued at $p$.

In effect, this is equivalent to differentiating $\bar W_{aLbL}$ using the following relations:
\begin{equation}\label{S2_frame}\begin{split} \tilde{\nabla}_a L&=e_a\\
\tilde{\nabla}_a \underline L&=-\frac{1}{2} e_a\\
\tilde{\nabla}_a e_b&=(\underline{L}-\frac{1}{2}L)\tilde{\sigma}_{ab}.\\
\end{split}\end{equation}

Therefore,
\[
\begin{split}
\tilde \nabla_c \alpha_{ab} =& \tilde \nabla_c  \bar W_{aLbL}\\
=& \bar W(\tilde{\sigma}_{ca} \underline{L}, L , e_b, L) + \bar W(e_a, e_c, e_b, L)+\bar W(e_a, L, \tilde{\sigma}_{cb}\underline L, L)+
 \bar W(e_a, L, e_b, e_c) \\
=& \tilde{\sigma}_{ca} \beta_b+\tilde{\sigma}_{cb} \beta_a+\bar W_{acbL}+\bar W_{bcaL}.
\end{split}
\] Substituting \eqref{relations} gives the desired formula. 
Other formulae can be derived similarly, for example
\[
\begin{split}
\tilde \nabla_a \beta_b =& \tilde \nabla_a  \bar W_{Lb L \underline L}\\
=& \bar W_{a b L \underline L} +\bar W_{L\underline{L} L \underline L} \tilde  \sigma_{ab}  +  \bar W_{Lb a \underline L} - \frac{1}{2} \bar W_{Lb L a} \\
= &-\frac{3}{4} \sigma \epsilon_{ab}  + \frac{3}{2} \rho \tilde \sigma_{ab} - \frac{1}{2} \alpha_{ab}.
\end{split}
\]
\end{proof}

Contracting with respect to $\tilde{\sigma}_{ab}$ and $\epsilon_{ab}$, we obtain the following formulae:
\begin{lemma}
\begin{equation}\label{Weyl_derivatives2}
\begin{split}\tilde \nabla^a \alpha_{ab} = & 4 \beta_b, \epsilon^{ca}\nabla_c\alpha_{ab}=4\epsilon_{bd}\beta^d\\
\tilde \nabla^a \underline \alpha_{ab} = & 2 \underline \beta_b, \epsilon^{ca}\nabla_c\underline \alpha_{ab}=2\epsilon_{bd}\underline \beta^d\\
\tilde \nabla^a \beta_a = & 3 \rho, \epsilon^{ab}\tilde{\nabla}_a\beta_b=-\frac{3}{2}\sigma\\
\tilde \nabla^a \underline \beta_a = &\frac{3}{2} \rho, \epsilon^{ab}\tilde{\nabla}_a\underline\beta_b=\frac{3}{4}\sigma.\end{split}
\end{equation}
\end{lemma}
\begin{proof} We apply the following two identities:
\[\tilde{\sigma}^{ca}(\tilde{\sigma}_{ca}\tilde{\sigma}_{bd}+\tilde{\sigma}_{cb}\tilde{\sigma}_{ad}+\epsilon_{ca}
\epsilon_{bd}+\epsilon_{cb}
\epsilon_{ad})=4\sigma_{bd}\]
\[{\epsilon}^{ca}(\tilde{\sigma}_{ca}\tilde{\sigma}_{bd}+\tilde{\sigma}_{cb}\tilde{\sigma}_{ad}+\epsilon_{ca}
\epsilon_{bd}+\epsilon_{cb}
\epsilon_{ad})=4\epsilon_{bd}.\]
\end{proof}
In particular, it follows that $\tilde{\Delta}\rho=-6\rho$ and $\tilde{\Delta}\sigma=-6\sigma$. Moreover, we obtain the following lemma from equation \eqref{Weyl_derivatives1} and equation \eqref{Weyl_derivatives2}.
\begin{lemma}
\begin{align}
\label{gradient_alpha2}\tilde \nabla_c |\alpha|^2 = & 8 \alpha_{cb} \beta^b\\
\label{delta_alpha2}\tilde \Delta   |\alpha|^2 =  & 32 |\beta|^2 - 4|\alpha|^2
\end{align}
\end{lemma}
\begin{proof}
We compute 
\[ 
\begin{split}
\tilde \nabla_c |\alpha|^2 =& 2 (\tilde \nabla_c \alpha_{ab}) \alpha^{ab}\\
=& 2 (\tilde{\sigma}_{ca}\tilde{\sigma}_{bd}+\tilde{\sigma}_{cb}\tilde{\sigma}_{ad}+\epsilon_{ca}
\epsilon_{bd}+\epsilon_{cb}
\epsilon_{ad})\beta^d  \alpha^{ab} \\
 = & 8 \alpha_{cb} \beta^b.
\end{split}
 \]
Equation \eqref{delta_alpha2} follows from computing the divergences of the two sides of equation \eqref{gradient_alpha2} with the help of equation  \eqref{Weyl_derivatives1} and equation \eqref{Weyl_derivatives2}.
\end{proof}
We obtain the following two identities.
\begin{cor}
\begin{align}
\label{relation_alpha_beta_1}\int _{S^2} |\alpha|^2 dS^2 =& 8 \int_{S^2} |\beta|^2 dS^2\\
\label{relation_alpha_beta_2}\int _{S^2}\tilde X^i |\alpha|^2 dS^2 =& 16 \int_{S^2}\tilde X^i |\beta|^2 dS^2
\end{align}
\end{cor}
\begin{proof}
Equation \eqref{relation_alpha_beta_1} follows from integrating equation  \eqref{delta_alpha2} on $S^2$ with the standard metric. Similarly, equation \eqref{relation_alpha_beta_2} follows from multiplying equation  \eqref{delta_alpha2} with $\tilde X^i$ and then integrating on $S^2$.
\end{proof}
Furthermore, the covariant derivative in the spacetime $N$ at $p$ in the direction of $L$ is denoted by the symbol $D$. For example, 
\[ D\alpha_{ab} = \nabla^N_L W(e_a,L,e_b,L)(p).   \]
$D\alpha_{ab}$ is also considered as a tensor on $S^2$ through the limiting process and its covariant derivatives with respect to the standard metric $\tilde{\sigma}_{ab}$ can be computed in the same manner. Relations similar to equation \eqref {Weyl_derivatives2} hold among $D$ of the Weyl curvature components. 
\begin{lemma}\label{D_divergence}
\begin{equation}\begin{split}
 \tilde \nabla^a D \beta_a =  & 4 D \rho\\
 \tilde \nabla^a D^2 \beta_a =  & 5 D^2 \rho\\
\tilde\nabla^a(D\alpha_{ab})= & 5D\beta_b\\
\tilde\nabla^a(D^2\alpha_{ab})=& 6 D^2\beta_b\end{split}
\end{equation}
\begin{proof}
We compute
\[
\begin{split}
\tilde \nabla^a D \beta_a =& \tilde \sigma^{ab}\tilde \nabla_a (\nabla^N_L  \bar W_{Lb L \underline L})\\
=& \tilde \sigma^{ab} \nabla_a \bar W_{Lb L \underline L}+\nabla^N_L \tilde  \nabla^a \bar W_{La L \underline L} \\
=& 4  \nabla^N_L \bar W_{L\underline L L \underline L}.
\end{split}
\]
Similarly, 
\[
\begin{split}
\tilde \nabla^a D^2 \beta_a =& \tilde \sigma^{ab}\tilde \nabla_a (\nabla^N_L \nabla^N_L \bar W_{Lb L \underline L})\\
=& \tilde \sigma^{ab}( \nabla^N_{a}\nabla^N_{L}\bar W_{Lb L\underline L}+\nabla^N_{L}\nabla^N_{a}\bar W_{Lb L\underline L}+\nabla^N_{L}\nabla^N_{L}\tilde \nabla_a \bar W_{LbL\underline L})\\
=&\tilde \sigma^{ab}  (\nabla^N_{a}\nabla^N_{L}-\nabla^N_{L}\nabla^N_{a})\bar W_{Lb L\underline L}+ 5  \nabla^N_{L}\nabla^N_{L}\bar W_{L\underline L L\underline L}.
\end{split}
\]
It suffices to show that 
\[\tilde \sigma^{ab}  (\nabla^N_{a}\nabla^N_{L}-\nabla^N_{L}\nabla^N_{a})\bar W_{Lb L\underline L} =0  .\]
We compute 
\[
\begin{split}
 &\tilde \sigma^{ab} (\nabla^N_{a}\nabla^N_{L}-\nabla^N_{L}\nabla^N_{a})\bar  W_{Lb L\underline L}\\
= &\tilde \sigma^{ab} ( \bar  W_{aLL\alpha} \bar  W_{\alpha b L \underline L} +\bar  W_{aLb\alpha} \bar  W_{L\alpha  L \underline L}+\bar W_{aLL\alpha} \bar  W_{L b \alpha \underline L}+ \bar  W_{aL\underline L\alpha} \bar W_{L  bL \alpha }) \\
= &\tilde \sigma^{ab}(\bar  W_{aLL\alpha} \bar  W_{\alpha b L \underline L} + \bar  W_{aLb\alpha}\bar  W_{L\alpha  L \underline L})\\
=&-\tilde \sigma^{ab}(\bar W_{aLL\underline L} \bar  W_{L b L \underline L} )+ \bar  W_{\underline L L  L \alpha}\bar  W_{L\alpha  L \underline L}\\
=& \tilde \sigma^{ab}(-\bar  W_{aLL\underline L} \bar  W_{L b L \underline L} + \bar  W_{\underline L L  L a} \bar W_{Lb  L \underline L})\\
= & 0.
\end{split}
\]
The other two relations can be derived similarly. 
\end{proof}
\end{lemma}
\begin{remark}
It is also useful to evaluate the Weyl curvature tensor at the point $p$ on 
 \[ e_r=\tilde{X}^i \frac{\partial}{\partial x^i}.\]
For example, 
\[\bar W_{rabr}=\bar W(\tilde{X}^i \frac{\partial}{\partial x^i},   \frac{\partial \tilde{X}^j}{\partial u^a},\frac{\partial \tilde{X}^k}{\partial u^b}, \tilde{X}^l\frac{\partial}{\partial x^l}).\] 
\end{remark}
\begin{lemma}
We have the following expansions for the Weyl curvature tensor:

\begin{equation}\begin{split}\label{exp_LaLN} W_{LaL\underline L}&= r  \beta_a +r^2 D\beta_a +\frac{1}{2} r^3 D^2\beta_a+O(r^4) \\
 W_{L\underline LL\underline L  } &=\rho + r D \rho+r^2 [ \frac{1}{2} D^2 \rho -\frac{1}{3} |\beta|^2]+O(r^3). \end{split}
\end{equation}
\begin{proof}

We compute the expansion for $W_{LaL\underline L}$.
   \begin{equation} \label{R_{rarn}}
   \begin{split}
\partial_r W_{LaL\underline L} ={} & \nabla^N_{L}W_{LaL\underline L} - l_{a}^cW_{LcL\underline L } - \eta^b W_{LaLb}.  \\
                                                          \end{split}
\end{equation}

This is equivalent to
\begin{equation}\label{LaLN1}r\partial_r(r^{-1}W_{LaL\underline L})=\nabla^N_L W_{LaL\underline L}-(l_a^c+r^{-1}\delta_a^c)W_{LcL\underline L}-\eta^bW_{LaLb}.\end{equation}
Using the expansion of $l_a^c$ in equation \eqref{l_contracted} and the expansion of the curvature tensor in equation \eqref{exp_R_LaLN}, we have 
\begin{equation}\label{LaLN2} - (l_{a}^c+r^{-1}\delta_a^c)W_{LcLN } - \eta^b W_{LaLb}=O(r^3).
\end{equation}
We differentiate $\nabla^N_{L}W_{LaL\underline L}$ again and get
\begin{align*}
       & \partial_r(\nabla^N_{L}W_{LaL\underline L} )
=  \nabla^N_{L} \nabla^N_{L} W_{LaL\underline L} -l_a^c\nabla^N_{L}W_{LcL\underline L} -\eta^b\nabla^N_{L}W_{LaLb}. 
\end{align*}
This is equivalent to
\[r\partial_r(r^{-1}(\nabla^N_L W)_{LaL\underline L})=rD^2 \beta_a+O(r^2).\]
Thus 
\begin{equation}\label{LaLN3}\nabla^N_{L}W_{LaL\underline L} = r D \beta_a+r^2 D^2 \beta_a+O(r^3) .
 \end{equation}

Using equation \eqref{LaLN2} and equation \eqref{LaLN3} in equation \eqref{LaLN1}, we rewrite

\[r\partial_r(r^{-1}W_{LaL\underline L})=r D \beta_a+r^2D^2 \beta_a +O(r^3).\]
Integrating this equation, we obtain
\[ W_{LaL\underline L}= r \beta_a +r^2 D \beta_a+\frac{1}{2}r^3 D^2 \beta_a+O(r^4) .\]

For $ W_{L\underline LL\underline L}$, we have
\begin{equation}\label{LNLN}
\partial_r W_{L\underline LL\underline L}  =  \nabla^N _{L} W_{L\underline LL\underline L} -2 W_{LaL\underline L} \eta^a .  \\
\end{equation}
We compute
\[ -2 W_{LaL\underline L} \eta^a=  - \frac{2}{3} r |\beta|^2+O(r^2)\]
and \[\nabla ^N_{L} W_{L\underline LL\underline L}=D \rho+r D^2 \rho+O(r^2).\]
It follows that
\[  W_{L\underline LL\underline L  } = \rho+ r D \rho+ \frac{1}{2}r^2 \left (D^2 \rho -\frac{2}{3}|\beta|^2 \right)+O(r^3).\]
\end{proof}

\end{lemma}
\begin{lemma} \label{data}
We have the following expansions for  $\sigma^{ab} l_{ab}$, $\sigma^{ab} n_{ab}$ and $\eta_a$.
\begin{equation}
\sigma^{ab} l_{ab} = - \frac{2}{r} +\frac{  1}{45} r^3 |\alpha|^2 +O(r^4)
\end{equation}
\begin{equation}
\sigma^{ab} n_{ab} = \frac{1}{r} + r(\sigma^{ab} n_{ab})^{(1)}+r^2 (\sigma^{ab} n_{ab})^{(2)}+r^3(\sigma^{ab} n_{ab})^{(3)}+O(r^4)
\end{equation}
and
\begin{equation}\label{exp_eta} \eta_a = \frac{r^2}{3} \beta_a +\frac{r^3}{4}   D\beta_a +r^4 [\frac{1}{10}   D^2 \beta_a  -\frac{1}{45} \alpha_{ab} \beta^b ] +O(r^5), \end{equation}
where
\begin{equation}
\begin{split}
(\sigma^{ab} n_{ab})^{(1)}=&\rho \\
(\sigma^{ab} n_{ab})^{(2)}=&\frac{2}{3}D\rho\\
(\sigma^{ab} n_{ab})^{(3)}=&   \frac{3}{8}D^2\rho+\frac{1}{30}|\alpha|^2- \frac{11}{45}|\beta|^2 .
\end{split}
\end{equation}
\end{lemma}
\begin{proof}

We  rewrite $l_{ab}$ as 
\begin{align*}
l_{ab} =& -r \tilde \sigma_{ab} -\frac{2}{3}r^3\alpha_{ab}+ O(r^4)\\
=& (-r \tilde \sigma_{ab} -\frac{1}{3}r^3\alpha_{ab}) - \frac{1}{3}r^3\alpha_{ab}+ O(r^4).
 \end{align*}
Hence, $\hat l_{ab}$, the traceless part of $l_{ab}$, is given by
\begin{equation}
\hat l_{ab} =-\frac{1}{3}r^3\alpha_{ab} + O(r^4).
\end{equation}
It follows that
\begin{equation}  \sigma^{ab} l_{ab} = - \frac{2}{r} +\frac{  r^3}{45} |\alpha|^2    +O(r^4). \end{equation}

Next we compute $\eta_a$. Let
\[ \eta_a = r^2 \eta_a^{(2)} +r^3 \eta_a^{(3)} +r^4 \eta_a^{(4)} +O(r^5). \]
From Lemma \ref{div_non}, we have
\[  \eta_a = \frac{1}{3}r^2 \beta_a+O(r^3). \]
Equation \eqref{eq_eta} is equivalent to
\[r^{-1}\partial_r(r\eta_a)=W_{LaL\underline L}+(l_a^b+r^{-1}\delta_a^b)\eta_b.\]
By equation \eqref{exp_LaLN}, the right hand side can be expanded into
\[r \beta_a+r^2 D\beta_a+r^3[\frac{1}{2}D^2\beta_a-\frac{1}{9}\alpha_{ab}\beta^b]+O(r^4).\]
Integrating, we obtain
\begin{align*}
 \eta_a^{(3)}  = &\frac{1}{4}   D\beta_a  \\
  \eta_a^{(4)} = &\frac{1}{10}   D^2 \beta_a -\frac{1}{45}\alpha_{ab}\beta^b.
\end{align*}

For $n_{ab}$, we start with equation \eqref{eq_n}. It is equivalent to
\[  r\partial_r (r^{-1} n_{ab}) = W_{Lab\underline L}  - (l_b^c+r^{-1}\delta_b^c) n_{ac} +  \nabla_a  \eta_b - \eta_a\eta_b.\]
The equation becomes
\[
\begin{split}
  r\partial_r (r^{-1} n_{ab}) =&r^2[\bar W_{Lab\underline L}  -\frac{1}{6} \bar W_{LabL} +   \tilde \nabla _a  \eta^{(2)}_b]+O(r^3)\\
=&r^2 \rho \tilde \sigma_{ab}+O(r^3).
\end{split}
\]
Integrating, we obtain
\[ n_{ab} =\frac{1}{2} r \tilde{\sigma}_{ab} + \frac{1}{2} r^3  \rho \tilde \sigma_{ab} + O(r^4).  \]

Lastly, we deal with equation \eqref{eq_trn}  for $\sigma^{ab}n_{ab}$. We decompose
\begin{equation}
\begin{split}
l^{ab} n_{ab} = & l_{ab} \sigma^{ac} \sigma^{bd} n_{cd} \\
                     =& (l_{ab}+\frac{\sigma_{ab}}{r}) \sigma^{ac} \sigma^{bd} (n_{cd} - \frac{\sigma_{cd}}{2r}) + \frac{1}{2} r^{-1}\sigma^{ab} l_{ab} - r^{-1}\sigma^{ab}n_{ab} + r^{-2} .       \\
                     \end{split}
\end{equation}
Thus equation \eqref{eq_trn} is equivalent to
\[r^{-1}\partial_r (r\sigma^{ab} n_{ab}) = r^{-2}+\frac{1}{2}r^{-1}\sigma^{ab} l_{ab}+  (l^{ab}+r^{-1}\sigma^{ab}) (n_{ab}-\frac{1}{2} r^{-1}\sigma_{ab})- \eta_a\eta^a + W_{L\underline LL\underline L} + div_{\sigma} \eta. \]
Notice that \[
\begin{split}
l_{ab}+r^{-a}\sigma_{ab}=&-\frac{1}{3}r^3\alpha_{ab}+O(r^4)\\
n_{cd}-\frac{1}{2} r^{-1}\sigma_{cd}
= &r^3( \frac{1}{2}\rho \tilde \sigma_{ab}- \frac{1}{6}\alpha_{ab})+O(r^4).
\end{split}
\]
We have
\[\begin{split}
r^{-1}\partial_r (r \sigma^{ab} n_{ab}) 
= & r^ 2 \left [\frac{1}{15}|\alpha|^2-  \frac{1}{9} |\beta|^2   \right ]   + div_{\sigma} \eta+ W_{L\underline LL\underline L} +O(r^3).
\end{split}
\]Integrating this equation, we obtain the expansion for $\sigma^{ab}n_{ab}$.

\[
\begin{split}
\sigma^{ab} n_{ab} =& \frac{1}{r} + \frac{r}{2}[(\rho + (div_{\sigma} \eta)^{(0)}] + \frac{r^2}{3}[ D \rho+ (div_{\sigma} \eta)^{(1)}]+ \frac{r^3}{4} \Big [ \frac{1}{2} (D^2 \rho -\frac{2}{3} |\beta|^2) \\
&+ (div_{\sigma} \eta)^{(2)}-\frac{1}{9}|\beta|^2+ \frac{1}{15}  |\alpha|^2  \Big ] + O(r^4).
\end{split}
\]

We compute
\begin{equation}\label{divergence_null_connection}
\begin{split}
(div_{\sigma} \eta)^{(0)}=&\frac{1}{3}\tilde{\nabla}^a(
\beta_a)=\rho \\
(div_{\sigma} \eta)^{(1)}=&\frac{1}{4} \tilde{\nabla}^a D \beta_a = D \rho \\
(div_{\sigma} \eta)^{(2)}=&\frac{1}{10} \tilde{\nabla}^a  (D^2 \beta_a ) -\frac{2}{15}\tilde{\nabla}^a ( \alpha_{ab} \beta^b)\\
=&\frac{1}{2}  D^2\rho+\frac{1}{15}( |\alpha|^2- 8 |\beta|^2).
\end{split}
\end{equation}
Equation \eqref{divergence_null_connection} follows from the expansion of $\eta$ and the following expansion for $\gamma_{ab}^c$:
\begin{equation}\label{gamma} \gamma_{ab}^c = \tilde \gamma_{ab}^c +r^2 \gamma_{ab}^{(2)c}+O(r^3) \end{equation}
where $ \tilde \gamma_{ab}^c $ is the  Christoffel symbols for $\tilde \sigma_{ab}$ and $\gamma_{ab}^{(2)c}= -\frac{1}{6} \tilde\sigma^{cd}(\tilde \nabla_a \alpha_{db}+ \tilde \nabla_b \alpha_{ad} - \tilde \nabla_d \alpha_{ab})$.
\end{proof}
\section{Small sphere limit of the quasi-local energy in spacetimes with matters} \label{sec_limit_non}
Recall that the Einstein equation is 
\[ R_{\alpha \beta} - \frac{1}{2} R g_{\alpha \beta} = 8 \pi T_{\alpha \beta}. \]
For the small sphere limit of the Wang--Yau quasi-local energy, we show that the leading order term of the quasi-local energy is precisely  the stress-energy density with the order $O(r^3)$. 
This is true with respect to any isometric embeddings with time functions of the order $O(r^3)$. Moreover, among such embeddings, there is a solution to the 
optimal embedding equation when the observer is chosen correctly.
\subsection{The optimal embedding equation}
We study the optimal embedding equation to determine the leading order of the time function. 
As in \cite{Chen-Wang-Yau1}, we will restrict ourselves to isometric embeddings close to the embedding into $\R^3$. Namely, we suppose that the embedding is given by
$X=(X_0,X_1,X_2,X_3)$ and an observer $T_0$ with the following expansion
\begin{equation}\label{assume}
\begin{split} 
X_0  = & \sum_{i=2}^{\infty}X_0^{(i)}r^i \\
X_k = &  r \tilde X^k+ \sum_{i=3}^{\infty}X_k^{(i)}r^i  \\
T_0= & (a^0,-a^i)+\sum_{i=1}^{\infty}T_0^{(i)}r^i. 
\end{split}
\end{equation}
where $X_0^{(i)}$, $X_k^{(i)}$ and $T_0^{(i)}$ are independent of $r$. 

Let $\tau = -X \cdot T_0$ be the time function. The optimal embedding equation is 
\begin{equation} \label{optimal}
div(f \nabla \tau) - \Delta [ \sinh^{-1} (\frac{\Delta \tau f}{|H||H_0|})]  =div_{\sigma}  \alpha_{H_0} - div_{\sigma} \alpha_{H}, \end{equation}
where
\[ f = \frac{\sqrt{|H_0|^2 +\frac{(\Delta \tau)^2}{1+ |\nabla \tau|^2}} - \sqrt{|H|^2 +\frac{(\Delta \tau)^2}{1+ |\nabla \tau|^2}} }{ \sqrt{1+ |\nabla \tau|^2}}. \] 
Assuming equation (\ref{assume}) and using Lemma \ref{non_physical_data}  for $|H|$ and Lemma  4 of \cite{Chen-Wang-Yau1} for $H_0$, we have
\begin{equation}\label{compatibility_verify} |H|  = \frac{2}{r} +O(r)  \text{ and } |H_0|  = \frac{2}{r} +O(r). \end{equation}
As a result, $f=O(r)$ and 
\[  div(f \nabla \tau) - \Delta [ \sinh^{-1} (\frac{\Delta \tau f}{|H||H_0|})] = O(1).   \]
From Lemma \ref{non_physical_data}, we have  $div_{\sigma} \alpha_{H} =O(1)$.  For $div \alpha_{H_0} $, we use Lemma  5 of \cite{Chen-Wang-Yau1} to conclude
\[  div_{\sigma} \alpha_{H_0}= \frac{1}{2}r^{-1} \tilde \Delta(\tilde \Delta +2)X_0^{(2)}  + O(1) .\]
However, since all other terms in equation (\ref{optimal}) are at most $O(1)$, we conclude that
\[X_0  =  \sum_{i=3}^{\infty}X_0^{(i)}r^i\]
and 
\[  div_{\sigma} \alpha_{H_0} = \frac{1}{2} \tilde \Delta ( \tilde \Delta +2 ) X_0^{(3)} + O(r). \]
\begin{remark}
We can also deduce that $X_0^{(2)}$ should vanish from the point of view of minimizing energy. Without loss of generality, we may assume  $X_0^{(2)}$ is perpendicular to the kernel of the operator $\tilde \Delta(\tilde \Delta +2)$, up to changing lower order terms in $T_0$.
Let $\widehat X_r$ be the isometric embedding of $\Sigma_r$ into the $\R^3$. Following the discussion of \cite{Chen-Wang-Yau1} for the second variation of the Wang--Yau quasi-local energy and the order of $div_{\sigma} \alpha_H$ above, we conclude
\[  E(\Sigma_r,X_r,T_0) = E(\Sigma_r, \widehat X_r, T_0) + \frac{r^{3}}{4} \int_{S^2} X_0^{(2)}   \tilde \Delta(\tilde \Delta +2)X_0^{(2)} dS^2 + O(r^{4}),\]
where the second term is strictly positive unless $X_0^{(2)}$ vanishes. 
\end{remark}
From equation \eqref{compatibility_verify}, it follows that 
\[  \lim_{r \to 0} \frac{|H|}{|H_0| }=1 \]
 and we apply Theorem 2.1 of \cite{Wang-Yau3} to evaluate the limit of the quasi-local energy.  The limit of the quasi-local energy is the linear function dual to the four-vector $(e,p^1,p^2,p^3)$ where
 \[
\begin{split}
 e = &\frac{1}{8 \pi}  \int_{\Sigma_r}(|H_0|-|H|) d\Sigma_r  \\
 p^i =& \frac{1}{8 \pi}  \int_{\Sigma_r}  X^i ( div_{\sigma} \alpha_{H} - div_{\sigma} \alpha_{H_0} )d \Sigma_r. 
\end{split}
\]   

Following the argument in \cite{Chen-Wang-Yau1}, we show that the limit is independent of $X_0^{(3)}$ as follows. It is easy to see that $e$ and $p^i$ are $O(r^3)$. Furthermore, 
$|H_0|$ is the same up to an error of $O(r^3)$ for any isometric embedding with time functions of the order $O(r^3)$. Hence the leading order term of $e$ is independent of 
the isometric embedding. For $p^i$, we have

\begin{align*}
   &\frac{1}{8 \pi}  \int_{\Sigma_r}  X_i ( div_{\sigma} \alpha_{H} - div_{\sigma} \alpha_{H_0} )d \Sigma_r \\
=&  \frac{1}{8 \pi}  \int_{\Sigma_r}  r \tilde X^i ( div_{\sigma} \alpha_{H} -\frac{1}{2} \tilde \Delta ( \tilde \Delta +2 ) X_0^{(3)}  ) d \Sigma_r +O(r^4).\\
=&  \frac{1}{8 \pi}  \int_{\Sigma_r}   X^i ( div_{\sigma} \alpha_{H} )d \Sigma_r +O(r^4).
\end{align*}
Hence it suffices to consider the isometric embedding into $\R^3$ to evaluate the limit.
\subsection{Small sphere limit}
\begin{theorem} \label{thm_small_non}
Let $\Sigma_r$ be the family of affine parameter $r$ from $p$, normalized by the unit timelike vector $e_0$. For any family of isometric embedding $X_r$ of $\Sigma_r$ into $\R^{3,1}$ such that $X_0=O(r^3)$, the limit of the quasi-local energy $E(\Sigma_r,X_r,T_0)$ as $r$ goes to $0$ satisfies 
\[ \lim_{r \to 0} r^{-3}E(\Sigma_r, X_r,T_0) = \frac{4\pi}{3}T(e_0, T_0). \]
\end{theorem}
\begin{proof}
From the previous subsection, it suffices to evaluate
\[ e = \frac{1}{8 \pi}  \int_{\Sigma_r}(H_0-|H|) d \Sigma_r \]
\[ p^i = \frac{1}{8 \pi}  \int_{\Sigma_r} X^i div_{\sigma} \alpha_H  d\Sigma_r.\]
where $H_0$ is the mean curvature of the isometric embedding into $\R^3$.
\begin{lemma}
For the energy component of the limit, we have
\[\frac{1}{8\pi}\int_{\Sigma_r}(H_0-|H|)  d\Sigma_r= \frac{4}{3}r^3  T(e_0,e_0)+O(r^4). \]
\end{lemma}
\begin{proof}
A similar equality is proved in \cite{yu}. Namely, 
\[\frac{1}{8\pi}\int_{\Sigma_r}(2\sqrt{K}-|H|)  d\Sigma_r= \frac{4}{3}r^3 T(e_0,e_0)+O(r^4) \]
where $K$ is the Gaussian curvature of $\Sigma_r$.  Using the Gauss equation for the image of the isometric embedding into $\R^3$, we conlcude that
\[  2 \sqrt{K} - H_0 = O(r^2) .\]
This finishes the proof of the lemma.
\end{proof}
Next we compute $p^i$.  Using the expansion for $div_\sigma \alpha_H$ from Lemma \ref{non_physical_data}, we have
\begin{align*}
  p^i =  r^3 \int_{S^2} &   {\Big [}  \tilde\Delta[  \frac{1}{2} \bar R_{L\underline LL\underline L} + \frac{1}{6} \bar Ric(L,L)  + \frac{1}{3}\bar Ric(L,\underline L)] \\ 
&  -\bar R_{L\underline LL\underline L} - \frac{1}{3}\bar Ric(L,\underline L)- \frac{1}{6}\bar Ric(L,L){\Big ]} \tilde X^idS^2 + O(r^{4}) \\
 =r^3 \int_{S^2} &   {\Big [}  -2\bar R_{L\underline LL\underline L} - \bar Ric(L,\underline L)- \frac{1}{2}\bar Ric(L,L){\Big ]} \tilde X^idS^2 + O(r^{4}) 
\end{align*}
where we apply integration by parts for the last equality.

To evaluate the above integral, we switch to the orthogonal frame $\{e_0 , e_i\}$. We have
\begin{align*}
  \int_{S^2}  {\Big [}  -2\bar R_{L\underline LL\underline L} - \bar Ric(L,\underline L)- \frac{1}{2}\bar Ric(L,L){\Big ]} \tilde X^idS^2  
=&  - \int_{S^2}   \bar Ric(e_0,e_j)\tilde {X}^j \tilde {X}^idS^2 \\
= & \frac{-4 \pi}{3}\bar Ric(e_0,e_i).
\end{align*}
This finishes the proof since $ \bar Ric(e_0,e_i) = 8 \pi T_{0i}.$ 
\end{proof} 
From the above theorem and Section 4 of \cite{Chen-Wang-Yau1}, we conclude that the linearized optimal embedding  is invertible if $T (e_0 , \cdot)$ is timelike.
The second part of Theorem \ref{main_theorem_1} now follows from the first part and the results in Section 4 of \cite{Chen-Wang-Yau1} for the solution of the optimal embedding equation. The only formal difference is that we have an expansion in $r$ for small $r$ here rather than an expansion in $\frac{1}{r}$ for $r$ large. 
\section{Functions and integrations in terms of the Weyl curvature at $p$} \label{sec_weyl_at}
When we compute the small sphere limit in vacuum spacetimes, there are several functions, tensors and integrations on $S^2$ which appear repeatedly. We compute these quantities here for use in the later sections.

We define the functions $W_0$, $W_i$ and $P_k$ as follows: 
\begin{equation}\label{W_P}
\begin{split}
W_0=&  \tilde X^i \tilde X^j \bar W_{0i0j} =\rho\\
W_i=  & \tilde X^j \tilde X^k \bar W_{0kij}=\frac{1}{2}(\beta^b-2\underline{\beta}^b)\tilde{\nabla}_b \tilde{X}^i \\
P_k= & \frac{1}{15}\bar W_{0i0k} \tilde X^i-\frac{1}{6}W_0\tilde X^k=-\frac{1}{30}(\beta^a+2\underline{\beta}^a)\tilde{\nabla}_a \tilde{X}^k-\frac{1}{10} \rho\tilde{X}^k\\
\end{split}
\end{equation}
$W_i$ are $-6$-eigenfunctions and $P_k$ are $-12$-eigenfunctions of the standard Laplacian on $S^2$. $P_k$ will appear in the solution of the optimal isometric embedding equation, Lemma \ref{x03}.

\begin{lemma} \label{WPinner}  For $P_j$ defined above, we have
\[ \int_{S^2} W_0  \tilde {\nabla} \tilde {X}^i \cdot\tilde{\nabla}P_j \,\, dS^2 = 4\int_{S^2} W_0\tilde {X}^i P_j \, dS^2.\]
\end{lemma}
\begin{proof} Integrating by parts,
\[
  \int_{S^2} W_0 \tilde {\nabla} \tilde {X}^i \cdot \tilde {\nabla}P_j dS^2
  =- \int_{S^2} \tilde{\nabla}(\rho \tilde {\nabla} \tilde {X}^i ) P_j dS^2
 = 2\int_{S^2} W_0 \tilde {X}^i P_j dS^2- \int_{S^2} P_j  \tilde {\nabla} \tilde {X}^i \cdot \tilde {\nabla} W_0 dS^2.
\]
Moreover, 
\begin{align*} 
     \tilde {\nabla} \tilde {X}^i \cdot \tilde {\nabla}W_0 
=&2  \bar W_{0m0n}  \tilde {X}^m (\delta^{in}- \tilde {X}^n  \tilde {X}^i) \\
=&-2 W_0 \tilde {X}^i  +2\bar  W_{0m0i} \tilde {X}^m, \\
\end{align*} and $\int_{S^2} P_j \tilde{X}^m dS^2=0$.
\end{proof}

We also introduce $R_{ij}$ and $S_j$ which will appear in Lemma \ref{lemma_7_1} later. From the expansion of the induced metric $\sigma_{ab}$, we derive
\begin{equation}
\sigma^{(0) ab}=-\frac{1}{3}\alpha^{ab} \text{ and  }
\tilde{\sigma}^{ab} \gamma_{ab}^{(2)c} =-\frac{4}{3}  \beta^c. 
\end{equation}

$R_{ij}$ and $S_j$ are defined as follows:
\begin{equation}\label{R_S}
\begin{split}
R_{ij}=& -\frac{1}{3}\alpha^{ab} \tilde{X}_a^i \tilde{X}_b^j =\sigma^{(0) ab}\tilde{X}^i_a \tilde{X}^j_b \\
S_j=&-\frac{4}{3}  \beta^c \tilde X_c^j=\tilde{\sigma}^{ab} \gamma_{ab}^{(2)c} \tilde{X}^j_c.
\end{split}
\end{equation}

\begin{lemma}\label{harmonic_Sj} $R_{ij}$ and $S_j$ defined in \eqref{R_S} satisfy:
\[  \begin{split}
R_{ij}&= \frac{1}{3}[2\tilde X^i\tilde X^k \bar W_{0i0k}
 + 2\tilde {X}^j\tilde {X}^k \bar W_{0k0i} +  \tilde X^i\tilde X^j\tilde X^n(\bar W_{0inj}+\bar W_{0jni}) \\
& -2\bar W_{0i0j}- \rho \delta_{ij} -\rho \tilde X^i\tilde X^j-\tilde X^n(\bar W_{0inj}+\bar W_{0jni}) ]\\
 S_j &=\frac{1}{3}(-4   \bar W_{0j0n}\tilde X^n + 4 \tilde X^j W_0 + 4 W_j). \end{split}\]
\end{lemma}

\begin{proof} Direct computations. 

\end{proof}
We need the following lemma for the integrals of products of spherical harmonic functions.
\begin{lemma}\label{sphere} 

\[\begin{split}\int_{S^2}\tilde {X}^i \tilde {X}^j  dS^2&= \frac{4 \pi}{3} \delta_{ij} \\
\int_{S^2}\tilde {X}^i \tilde {X}^j \tilde {X}^k \tilde {X}^l  dS^2&= \frac{4 \pi}{15} \Delta_{ijkl}\\
\int_{S^2}\tilde {X}^i \tilde {X}^j \tilde {X}^k \tilde {X}^l \tilde {X}^m \tilde {X}^n dS^2&= \frac{4 \pi}{105} (\delta_{ij} \Delta_{klmn}+\delta_{ik} \Delta_{jlmn}+\delta_{il} \Delta_{jkmn}
+\delta_{im} \Delta_{jkln}+\delta_{in} \Delta_{jklm}),\end{split}\] where
$\Delta_{ijkl}=\delta_{ij} \delta_{kl} +\delta_{ik} \delta_{jl}+\delta_{il} \delta_{jk}$.
\end{lemma}
\begin{proof} We repeatedly use $\tilde{\Delta} \tilde{X}^i=-2 X^i$ and $\tilde{\nabla}\tilde{X}^i\cdot \tilde{\nabla} \tilde{X}^j=\delta_{ij}-\tilde{X}^i\tilde{X}^j$ to compute the Laplacian of the integrand and then integrate by parts. For example, 
the integration of \[\tilde{\Delta} (\tilde{X}^i\tilde{X}^j)=-6 \tilde{X}^i\tilde{X}^j+2\delta_{ij}\] gives the first formula. 

\end{proof}

%%%%%%%%%%%%%%%%%%%%%%%%%%%%%%%%%%%%%%%%%%%%%%%%%%%%%%%%%%%%%%%%%%%%%%%%%%%%%%%%%%%%%%%%%%%%%%%%%%%%

\section{The optimal isometric embedding} \label{sec_optimal_va}

Assuming equation (\ref{assume}) for the isometric embedding, we determine $ X_0^{(3)}$ and $ X_i^{(3)}$ from the optimal embedding equation in this section.
We show that  $ X_i^{(3)}$ is determined purely by the induced metric
via the isometric embedding. However, for each $T_0$, there is a corresponding solution 
$ X_0^{(3)}$ of the leading order term of the optimal embedding equation depending on the choice of $T_0$.  This is different from the non-vacuum case of the small sphere limit or the large sphere limit, where only one choice of the  $T_0^{(0)}$ would allow the  the leading order term of the optimal embedding equation to be solvable. 
\begin{lemma}
\label{xi3}
\[ \begin{split}X^{(3)}_i=-\frac{1}{3}\beta^c \tilde{\nabla}_c \tilde{X}^i+\frac{1}{2} \rho\tilde{X}^i=\frac{1}{3}
 \bar W_{0i0j}   \tilde {X}^j +\frac{1}{3} \bar W_{0kji}   \tilde{X}^j  \tilde{X}^k +\frac{1}{6} \tilde{X}^i \tilde{X}^j \tilde{X}^k \bar W_{0j0k}
 \end{split} \] 
satisfies the linearized isometric embedding equation
 \begin{equation} \label{linearized_iso_equ}
 \sum_i \partial_a \tilde {X}^i \partial_b X_i^{(3)} + \partial_b \tilde {X}^i \partial_a X_i^{(3)}    = \frac{1}{3} \alpha_{ab}.\end{equation}
\end{lemma}
\begin{proof}
To solve the linearized isometric embedding equation, we write
\[  X_i^{(3)} = N \tilde X^i + P^a \tilde X^i_a . \]
Differentiating, we have 
\begin{equation} \label{first_derivative_iso}
 (X_i^{(3)})_b = N_b \tilde X^i + N\tilde X^i_b+ \tilde \nabla_b P^a \tilde X^i_a - P_b \tilde X^i. \end{equation}
In terms of $N$ and $P^a$, equation \eqref{linearized_iso_equ} is
\begin{equation}\label{linear-iso-gauge}
2N \tilde \sigma_{ab} + \tilde \nabla_a P_b +\tilde \nabla_b P_a = \frac{1}{3} \alpha_{ab}.   \end{equation}
By \eqref{Weyl_derivatives1}, we check that $N=\frac{1}{2}\rho$ and $P^a=-\frac{1}{3} \beta^a$ satisfy the above equation. 
This finishes the proof of the lemma.
\end{proof}
\begin{lemma} 
Let $f$ be the function as defined in equation \eqref{density}. We have
\[ f= \frac{W_0}{a^0} r + O(r^{2}) . \]
\end{lemma}
\begin{proof}
\[ f  =  \frac{|H_0| - |H|}{ a^0} +O(r^2), \]
where
\[ |H_0| = \frac{2}{r}  +2 W_0  r +O(r^2) \,\,\, {\rm and } \,\, \, |H| = \frac{2}{r}   + W_0  r +O(r^2) \]
from the result in \cite{blk2}. 
\end{proof}
\begin{lemma}\label{x03}
For the observer $T_0= (a^0, -a^1,-a^2,-a^3)$, the solution of the optimal embedding equation gives
\[\begin{split} X_0^{(3)}= - \frac{1}{3}W_0+\frac{a^i}{a^0}P_i.\end{split}\]
\end{lemma}

\begin{proof}
We compute
\[ div (f \nabla \tau) - \Delta( \sinh^{-1}(\frac{f \Delta \tau}{|H||H_0|})) = \tilde \nabla^a ( f^{(1)} \tilde \nabla_a \tau^{(1)}) -\frac{1}{4} \tilde \Delta (  f^{(1)} \tilde \Delta \tau^{(1)}) +O(r). \]
After simplification, the right hand side, up to a term of $O(r)$ is $-6 f^{(1)}  \tau^{(1)}  + 2 \tilde \nabla  \tau^{(1)} \cdot \tilde \nabla   f^{(1)}$, or $60\frac{a^i}{a^0}P_i$, by the definition of $P_i$ in \eqref{W_P}.

Setting the Ricci curvature to $0$ in Lemma \ref{non_physical_data}, we conclude 
 \[ div_\sigma \alpha_H = - 4W_0 + O(r).\]
Recall that 
 \[ div_\sigma \alpha_{H_0} =\frac{1}{2} \tilde \Delta (\tilde \Delta +2) X_0^{(3)}+ O(r).\]

 The top order term of the optimal isometric embedding equation is thus
\[ \frac{1}{2} \tilde\Delta (\tilde \Delta+2) X_0^{(3)} =-4W_0+ 60\frac{a^i}{a^0}P_i.  \]
The lemma follows since $W_0$ is a $-6$-eigenfunction and $P_i$ are $-12$-eigenfunctions.
\end{proof}
\begin{cor}
For any isometric embedding into $\R^{3,1}$ with $O(r^3)$ time function, we have 
\[|H_0| = \frac{2}{r} + 2 W_0 r +O(r^2).  \]
\end{cor}
\begin{remark}
By Lemma 4 of \cite{Chen-Wang-Yau1}, the result is the same for any  isometric embedding into $\R^{3,1}$ with $O(r^3)$ time function. 
For the embedding into $\R^3$, this is computed in \cite{blk2}.
\end{remark}

For each choice of $T_0^{(0)}$, we will compute $8 \pi E(\Sigma_r,X_r(T_0),T_0)$ which is given by
\begin{equation} \label{energy_expression}
 \int_{\Sigma_r } f(1 + |\nabla \tau|^2) +(\Delta \tau)  \sinh^{-1} ( \frac{f \Delta \tau}{|H||H_0|})  d\Sigma_r + \int_{\Sigma_r} \tau div_{\sigma} \alpha_{H_0} d\Sigma _r - \int_{\Sigma_r} \tau div_{\sigma} \alpha_{H} d\Sigma_r.
\end{equation}
We evaluate the three integrals in the next three sections, respectively and put the results together in Section \ref{sec_eva_energy}. Then we minimize the energy among all $T_0^{(0)}$ in Section \ref{sec_energy_min}.

%%%%%%%%%%%%%%%%%%%%%%%%%%%%%%%%%%%%%%%%%%%%%%%%%%%%%%%%%%%%%%%%%%%%%%%%%%%%%%%%%%%%%%%%%%%%%%%%%%%%%%%%%%%%%
\section{The energy component} \label{sec_energy_com_va}
In this section, we evaluate the first integral in equation \eqref{energy_expression}: 
\[ \int_{\Sigma_r}  f(1 + |\nabla \tau|^2) +(\Delta \tau)  \sinh^{-1} ( \frac{f \Delta \tau}{|H||H_0|}) d\Sigma_r.\] 
It suffices to evaluate $ \int_{\Sigma_r}  f [1 + |\nabla \tau|^2 +  \frac{(\Delta \tau)^2}{|H||H_0|}]  d \Sigma_r$ since for $x$ small, 
\[  \sinh^{-1}(x) =x + O(x^3). \]

Denote the expansion of the physical data by  
\[  
\begin{split}
\sigma_{ab} = & r^2\tilde \sigma_{ab} +r^4 \sigma_{ab}^{(4)}+r^5 \sigma_{ab}^{(5)}+O(r^6)\\
|H| = & \frac{2}{r} + r h^{(1)} + r^2 h^{(2)}+ r^3 h^{(3)}+O(r^4) \\
\alpha_H = & r^2 \alpha_H^{(2)}+ r^3 \alpha_H^{(3)}+ r^4 \alpha_H^{(4)}+ O(r^{5}).
\end{split}
\]
Furthermore, for the embedding $X_r(T_0)$ from Section \ref{sec_optimal_va}, we have
\[  
\begin{split}
|H_0| = & \frac{2}{r} + r h_0^{(1)} + r^2 h_0^{(2)}+ r^3 h_0^{(3)}+O(r^4) \\
\alpha_{H_0} = & r^2 \alpha_{H_0}^{(2)}+ r^3 \alpha_{H_0}^{(3)}+ r^4 \alpha_{H_0}^{(4)}+ O(r^{5}).
\end{split}
\]

First we derive the following lemma.
\begin{lemma}\label{lemma_7_1}
\[  
\begin{split}
|\nabla \tau| ^2   = &\sum_{ij} a^i a^j(\delta^{ij}  -\tilde {X}^i \tilde {X}^j) +  g_1 r^2 + O(r^{3})\\
   (\Delta \tau)^2    =& 4\sum_{ij} a^i a^j(\tilde {X}^i \tilde {X}^j) r^{-2} +  g_2 + O(r),
\end{split}
\]
where
\[\begin{split}
g_1& =a^i a^j (R_{ij}+2 \tilde \nabla \tilde {X}^i \tilde \nabla  X_j^{(3)})+2a^0 a^i \tilde \nabla\tilde  X^i \tilde \nabla X_0^{(3)} \\
g_2& =4a^i a^j \tilde {X}^i( S_j-\tilde{\Delta} X_j^{(3)})-4a^0a^i \tilde {X}^i\tilde{\Delta} X_0^{(3)},
\end{split}\] and $R_{ij}$ and $S_j$ are defined in equation \eqref{R_S}. 
\end{lemma}
\begin{proof}
We have 
\[\tau= \sum_i a^i \tilde {X}^ir+( a^i X_i^{(3)}+a^0 X_0^{(3)})r^3+O(r^4)\]
since $T_0 = (a^0,-a^i) + O(r)$. As a result, 
\[ 
\begin{split}
   |\nabla \tau| ^2   = & \sum_{ij} a^i a^j(\delta^{ij}  -\tilde {X}^i \tilde {X}^j)  \\
  &+  r^2 \left  [ \sigma^{(0) ab}(a^i \tilde X_a^i)(a^j \tilde X_b^j)+2(a^i \tilde \nabla \tilde  X^i) (a^j \tilde \nabla X_j^{(3)}+a^0\tilde \nabla X_0^{(3)}) \right  ]+ O(r^{3}), \end{split}\] and the formula
  follows from equations \eqref{expansion_first}, \eqref{Weyl}, and \eqref{R_S}.
Similarly, 
\[  (\Delta \tau)^2    = 4\sum_{ij} a^i a^j(\tilde {X}^i \tilde {X}^j) r^{-2} +  4(a^i \tilde X^i)(a^j \tilde{\sigma}^{ab}\gamma_{ab}^{(2)c}\tilde X_c^j)-4(a^i \tilde X^i) ( a^j  \tilde \Delta X_j^{(3)}+a^0  \tilde \Delta X_0^{(3)})  + O(r), \]
where
\[\begin{split} & 4(a^i \tilde X^i)(a^j \tilde{\sigma}^{ab}\gamma_{ab}^{(2)c}\tilde{X}_c^j)-4(a^i \tilde X^i) (a^j \tilde{\Delta} X_j^{(3)}+a^0 \tilde{\Delta} X_0^{(3)})\\
= &4a^i a^j \tilde X^i S_j-4(a^i \tilde X^i) ( a^j \tilde{\Delta} X_j^{(3)}+a^0  \tilde{\Delta} X_0^{(3)})\\
= & 4a^i a^j \tilde X^i( S_j-  \tilde{\Delta} X_j^{(3)})-4a^0a^i \tilde X^i  \tilde{\Delta} X_0^{(3)}.\end{split}\]
\end{proof}
With the above lemma, we compute $f (1 + |\nabla \tau|^2 +  \frac{(\Delta \tau)^2}{|H||H_0|})$.
\begin{lemma}
\begin{align*}
&f (1 + |\nabla \tau|^2 +  \frac{(\Delta \tau)^2}{|H||H_0|})\\
= &  a^0 r(h_0^{(1)} - h^{(1)})  + a^0 r^2 (h_0^{(2)} - h^{(2)}) \\
&+ a^0 r^3 \Big [(h_0^{(3)} - h^{(3)})  + \frac{(W_0) (  g_1 + \frac{g_2}{4} -\frac{3}{2} W_0 \sum_{ij} a^ia^j\tilde {X}^i \tilde {X}^j   )  }{2(a^0)^2} \Big ]+O(r^4).
\end{align*}
\end{lemma}

\begin{proof} 
From Lemma \ref{lemma_7_1}, we have
\[ 1+ |\nabla \tau|^2 +\frac{(\Delta \tau)^2}{|H_0|^2} = (a^0)^2 + r^2( g_1+ \frac{g_2}{4}  - h_0^{(1)} \sum_{ij} a^i a^j \tilde {X}^i \tilde {X}^j ), \]
and thus
\[ |H_0| \sqrt{1+ |\nabla \tau|^2 +\frac{(\Delta \tau)^2}{|H_0|^2}}  =a^0 \left [ \frac{2}{r} + r(h_0^{(1)} + \frac{g_1+\frac{g_2}{4}  - h_0^{(1)} \sum_{ij } a^i a^j \tilde {X}^i \tilde {X}^j }{(a^0)^2} ) \right ] + O(r^2). \]
$|H| \sqrt{1+ |\nabla \tau|^2 +\frac{(\Delta \tau)^2}{|H|^2}} $ and $ 1+ |\nabla \tau|^2 +\frac{(\Delta \tau)^2}{|H_0||H|}$ can be computed similarly and $f (1 + |\nabla \tau|^2 +  \frac{(\Delta \tau)^2}{|H||H_0|})$ is equal to 
\begin{align*}
  &(\frac{4}{r}  + (h_0^{(1)} + h^{(1)}) r) [(h_0^{(1)} - h^{(1)}) r +(h_0^{(2)} - h^{(2)}) r^2 +(h_0^{(3)} - h^{(3)}) r^3 ] \times  \\
&\frac{1+|a|^2 + r^2 [g_1 + \frac{g_2}{4}  - \frac{(h_0^{(1)} + h^{(1)})}{2} \sum_{ij} a^ia^j \tilde {X}^i \tilde {X}^j ] }{a^0\{  \frac{4}{r}  +  r [h_0^{(1)} + h^{(1)} +  \frac{2g_1 + \frac{g_2}{2}  - 2(h_0^{(1)} + h^{(1)}) \sum_{ij} a^ia^j \tilde {X}^i \tilde {X}^j  }{(a^0)^2} ]  \}  }  .\\
\end{align*} Finally we plug in $h_0^{(1)}=2W_0$ and $h^{(1)}=W_0$.
\end{proof}
\begin{lemma}\label{lemmaenergy}

\[\begin{split}
  &\lim_{r\rightarrow 0} r^{-5} \int _{\Sigma_r} f (1 + |\nabla \tau|^2 +  \frac{(\Delta \tau)^2}{|H||H_0|})\,\, d \Sigma_r\\
=& a^0\int_{S^2}  (h_0^{(3)} - h^{(3)}) dS^2   -\frac{3a^ia^j}{4 a^0} \int_{S^2} W_0^2\tilde {X}^i\tilde {X}^j dS^2 \\
  &+ \frac{a^ia^j}{2a^0} \int_{S^2}(W_0)[R_{ij}+2\tilde{\nabla}\tilde {X}^i \cdot\tilde{\nabla}(X_j^{(3)} + P_j)+\tilde{X^i}(S_j-\tilde{\Delta}X_j^{(3)} +12P_j)]dS^2.
\end{split}\]
\end{lemma}
\begin{proof}
For the volume form, we have $ d\Sigma_r = r^2 dS^2 + O(r^5) $ from the expansion of metric in  Lemma \ref{non_physical_data}. As a result, it suffices to use $r^2 dS^2 $ for the volume form.

For the mean curvature in $\R^{3,1}$, 
\[  |H_0| = 2 \sqrt{K} +O(r^3)  \]
since $X_0 = O(r^3)$. Hence, using the result of \cite{yu}, we conclude that for $i=1,2$
\[  \int_{S^2}  (h_0^{(i)} - h^{(i)}) dS^2 = 0.\]
Thus
\[\begin{split}  \int_{S^2} W_0 (  g_1 + \frac{g_2}{4}) dS^2
 = & a^ia^j \int_{S^2}W_0[R_{ij}+2\tilde{\nabla}\tilde {X}^i  \tilde \nabla (X_j^{(3)}+P_j)+\tilde {X}^i(S_j-\tilde{\Delta}X_j^{(3)}+12P_j)]dS^2\\
 &-a^ia^0\int_{S^2} W_0(\frac{2}{3}\tilde{\nabla} \tilde {X}^i\tilde{\nabla}W_0+2\tilde {X}^iW_0) dS^2.\end{split}\]
The second integral on the right hand side vanishes by parity.
\end{proof}
\subsection{Computation of $\int  (h_0^{(3)} - h^{(3)})$}
Suppose $X$ is the isometric embedding of $\sigma$ into $\R^{3,1}$ of the form
\[X_0=r^3 X_0^{(3)}+ O(r^4)\]
\[X_i=r\tilde{X}^i+r^3 X_i^{(3)}+r^4 X_i^{(4)}+r^5 X_i^{(5)} +O(r^6),\]
where $X_0^{(3)}$ and $X_i^{(3)}$ are given by Lemma \ref{xi3} and Lemma \ref{x03}, respectively.

Let $X'$ be the isometric embedding of $\sigma$ into $\R^3$ where
\[(X_0)'=0\]
\[(X_i)'=r\tilde {X}^i+r^3 X_i^{'(3)}+r^4 X_i^{'(4)}+r^5 X_i^{'(5)} +O(r^6).\]
Let $A'$ be the second fundamental form the the embedding $X'$ and $\AA'$ be its traceless part. 
\[
\AA'_{ab}=r^3\AA_{ab}^{'(3)}+O(r^4)\\
\]

Suppose the Gauss curvature $K$ of $\sigma$ has the following expansion:
\begin{equation}\label{k_exp} 2\sqrt{K}=\frac{2}{r} + k^{(1)} r +  k^{(2)} r^2 + k^{(3)} r^3 +O(r^4).\end{equation}
We have
\begin{pro} \label{pe2}
The integral $\int_{S^2}  (h_0^{(3)} - h^{(3)}) dS^2$ can be written as follows:
\[\begin{split}\int_{S^2}  (h_0^{(3)} - h^{(3)}) dS^2= &   \frac{1}{2}\int_{S^2}|\AA^{'(3)}|^2_{\tilde{\sigma}} dS^2+\int (k^{(3)} - h^{(3)}) dS^2-\frac{2}{3}\int_{S^2} W_0^2 dS^2 \\
&-30\frac{a^ia^j}{(a^0)^2}\int_{S^2} P_i P_j dS^2.\end{split}\]
\end{pro}
\begin{proof}
We first rewrite 
\[  \int_{\Sigma_r} (|H_0| - |H|) d \Sigma_r = \int_{\Sigma_r}(|H_0| - 2\sqrt{K} )d \Sigma_r+ \int_{\Sigma _r} (2 \sqrt{K} -|H|)d \Sigma_r .\]
Using the result of \cite{yu}, we have
\[  \int_{\Sigma_r} (2 \sqrt{K} - |H| ) d \Sigma_r=  r^5 \int_{S^2} (k^{(3)}-h^{(3)})dS^2. \]
To evaluate $ \int_{\Sigma_r}(|H_0| - 2\sqrt{K} )d \Sigma_r$, recall that  $|H_0|^2$ is given by
\begin{equation}\label{h_0}|H_0|^2=-(\Delta X_0)^2+ \sum_{i=1}^3 (\Delta  X_i)^2 .\end{equation}

Let $H_0'$  be the mean curvature of $X'$. 
Similarly, $| H_0'|$ is given by
\begin{equation}\label{h'} | H_0'|^2  =  \sum_{i=1}^3 (\Delta  (X_i)')^2. \end{equation}
The Gauss equation reads \begin{equation}\label{gauss}  4 K = (H_0')^2 - 2 |\AA'|^2.\end{equation}

We compute from \eqref{h_0}, \eqref{gauss}, and \eqref{h'} that \[|H_0|^2-4K=2|\AA'|^2-(\Delta X_0)^2+\sum_{i=1}^3(\Delta X_i)^2-\sum_{i=1}^3(\Delta X_i')^2,\]
where \[\Delta X_0=\Delta(r^3 X_0^{(3)}+O(r^4))=r\tilde{\Delta} X_0^{(3)}+O(r^2)\]
and
\[\begin{split}\sum_{i=1}^3(\Delta X_i)^2-\sum_{i=1}^3(\Delta (X_i)')^2
=&\sum_{i=1}^3\Delta (X_i-X_i') \Delta(X_i+X_i')\\
=&-4r^2\tilde{X}^i\tilde{\Delta}(X_i^{(5)}-X_i^{'(5)})+O(r^3).\end{split}\]
As a result, we have
\[
\begin{split}
   &\int_{S^2}  (h_0^{(3)} - h^{(3)}) dS^2 \\
= & \frac{1}{2}\int_{S^2} |\AA^{'(3)}|^2_{\tilde{\sigma}} dS^2-\frac{1}{4}\int_{S^2}  (\tilde{\Delta}X_0^{(3)})^2 dS^2-\int_{S^2}  \tilde X^i \tilde{\Delta}(X_i^{(5)}-X_i^{'(5)}) dS^2+\int_{S^2}   (k^{(3)} - h^{(3)}) dS^2.
\end{split}
\]
To evaluate the second last terms, we need
\begin{lemma}\label{isom_higher}
If we choose $X_i^{(3)}=X_i^{'(3)}$ and $X_i^{(4)}=X_i^{'(4)}$, $X_i^{(5)}$ and $X_i^{'(5)}$ are related by
\[2\tilde{\nabla} \tilde{X}^i\cdot \tilde{\nabla} (X_i^{(5)}-X_i^{'(5)})=|\tilde{\nabla} X_0^{(3)}|^2.\]
\end{lemma}

\begin{proof}
This follows directly from the expansion of the metric and the isometric embedding equation.
\end{proof}
The proposition now follows from the expression of $X_0^{(3)}$ in Lemma \ref{x03}.
\end{proof}

\subsubsection{Computing $\int |\AA^{'(3)}|^2_{\tilde{\sigma}} dS^2$}

The notation in this subsubsection is slightly different from before. Let $X$ be  the isometric embedding  of $\sigma$ into $\R^3$ where
\[X=r\tilde {X}+r^3 X^{(3)}+O(r^4)\] and
\[\sigma=r^2\tilde{\sigma}+r^4 \sigma_{ab}^{(4)}+O(r^5).\] Let $\nu$ be the unit normal of $X$. Suppose
\[\nu=\tilde {X}+r^2\nu^{(2)}.\] 
We have \[\nu^{(2)}=-\langle \tilde {X}, X^{(3)}_a\rangle \tilde{\sigma}^{ab} \tilde {X}_b.\]

The second fundamental form $h_{ab}$ of $X$ has the following  expansion
\[h_{ab}=r\tilde{\sigma}_{ab}-r^3\langle \tilde {X}, \tilde{\nabla}_a\tilde{\nabla}_b X^{(3)}\rangle + O(r^{4}). \] The traceless part $\AA_{ab}$ of $h_{ab}$ has the following expansion
\[\AA_{ab}=r^3 \AA_{ab}^{(3)}+O(r^4).\]
\begin{lemma}
\[\AA^{(3)}_{ab}=(\tilde {X}^i_a\tilde {X}^j_b+\tilde {X}^i_b \tilde {X}^j_a)(-\frac{1}{4}W_0 \delta_{ij}-\frac{1}{2}W_{0i0j}). \]
\end{lemma}
\begin{proof}
We compute
\[\AA_{ab}^{(3)}=-\langle \tilde {X}, \tilde{\nabla}_a\tilde{\nabla}_b X^{(3)}\rangle+\frac{1}{2}\tilde{\sigma}_{ab} \langle \tilde {X}, \tilde{\Delta} X^{(3)}\rangle-\sigma_{ab}^{(4)}.\]
A direct computation shows
\[\langle \tilde {X}, \tilde{\nabla}_a\tilde{\nabla}_b X^{(3)}\rangle=(\tilde {X}^i_a\tilde {X}^j_b+\tilde {X}^i_b \tilde {X}^j_a)[\frac{1}{2} \delta_{ij}(-\frac{5}{6}W_0+\frac{2}{3}W_k \tilde {X}^k)+\frac{1}{6}(\bar W_{0i0j}-2\tilde {X}^k \bar W_{0jki})]\]
\[\langle \tilde {X}, \tilde{\Delta}X^{(3)}\rangle=-2 W_0+\frac{4}{3} W_k \tilde {X}^k\]
\[\begin{split}\sigma_{ab}^{(4)}=\frac{1}{3} (\tilde {X}^i_a \tilde {X}^j_b+\tilde {X}^j_a \tilde {X}^i_b )(W_{0i0j}+\frac{1}{2} W_0\delta_{ij} +\tilde {X}^kW_{0ikj}). \end{split}\]
Finally, we note that $ \sum_k W_k \tilde {X}^k =0. $
\end{proof}
This leads to the following.
\begin{lemma} \label{lemmaAA}
\[ \int_{S^2} |\AA^{(3)}|^2_{\tilde{\sigma}} dS^2 =3\int_{S^2} W_0^2 dS^2. \]
\end{lemma}
\begin{proof}
Using Lemma \ref{sphere}, it is clear that 
\[ \int_{S^2} W_0^2 dS^2 = \frac{8 \pi}{15} \sum_{ij} W_{0i0j}^2. \]
On the other hand,
\begin{align*}
  &   \int_{S^2}   |\AA^{'(3)}|^2_{\tilde{\sigma}} dS^2 \\
= & \frac{1}{4}\int_{S^2}  [\tilde \sigma_{ab} W_0 + \bar W_{0i0j}(\tilde  X^i_a \tilde X^j_b+\tilde X^i_b \tilde X^j_a) ] \tilde \sigma^{ac} \tilde \sigma^{bd} [\tilde \sigma_{cd} W_0 + \bar  W_{0k0l}(\tilde X^k_c \tilde X^l_d+ \tilde X^k_c \tilde X^l_d) ] dS^2\\
=& \frac{1}{4} \int_{S^2} \Big[ 2 W_0^2 + 4 W_0(\delta^{kl} - \tilde  X^k \tilde X^l)\bar W_{0k0l} +4\bar W_{0i0j}\bar W_{0k0l}(\delta^{ik}-\tilde X^i \tilde X^k)(\delta^{jl}-\tilde  X^j \tilde X^l) \Big] dS^2 \\
=& \frac{1}{4} \int_{S^2} \Big[2 W_0^2 + \frac{4}{3}\bar W_{0i0j}\bar W_{0i0j} \Big] dS^2 \\
=& 3\int_{S^2} W_0^2 dS^2.
\end{align*}
\end{proof}
\subsubsection{Computing $\int (k^{(3)} - h^{(3)}) dS^2$}
\begin{lemma}\label{lemma7.7}
\begin{equation}
 \int _{S^2} (k^{(3)} - h^{(3)}) dS^2 =-\frac{3}{4}\int_{S^2} W_0^2  d S^2-\frac{1}{60} \int _{S^2} |\alpha|^2 dS^2 + \frac{11}{45} \int _{S^2} |\beta|^2 dS^2.\end{equation}
\end{lemma}
\begin{proof}
First we compute $\int k^{(3)}  dS^2$. From equation \eqref{k_exp}, we have
\[ K= \frac{1}{r^2} + k^{(1)} + k^{(2)}r  + [k^{(3)} + \frac{(k^{(1)})^2}{4}] r^2  + O(r^3).\] 
We also have
\[ d \Sigma_r  = (r^2   - \frac{1}{180} r^6 |\alpha|^2) dS^2+O(r^7)\]
from the expansion of $\sigma^{ab} l_{ab}$ in Lemma \ref{data}. By the Gauss--Bonnet theorem $ \int_{\Sigma_r} K d \Sigma_r = 4 \pi$.
Collecting the $O(r^4)$ terms from the left hand side, we have
\[ \int _{S^2} k^{(3)} + \frac{(k^{(1)})^2}{4}  d S^2 = \frac{1}{180} \int_{S^2}  |\alpha|^2 dS^2. \]
Furthermore, $k^{(1)} = 2 W_0$. Hence
\[ \int _{S^2}  k^{(3)} dS^2 =-  \int _{S^2} W_0^2  d S^2 + \frac{1}{180} \int  _{S^2} |\alpha|^2 dS^2. \]
For $h^{(3)}$, we have
\[ h^{(3)} = (\sigma^{ab} n_{ab})^{(3)}- \frac{1}{90} |\alpha|^2 -  \frac{( \sigma^{ab} n_{ab}^{(1)} )^2}{4}.
 \]
Using Lemma \ref{data} and Lemma \ref{D_divergence}, we conclude
\begin{equation}
\begin{split}
 \int_{S^2} (k^{(3)} - h^{(3)}) dS^2 =& -\frac{3}{4}\int_{S^2} W_0^2  d S^2 + \frac{1}{60} \int_{S^2} |\alpha|^2 dS^2- \int _{S^2}(\sigma^{ab} n_{ab})^{(3)} dS^2  \\
=&-\frac{3}{4}\int_{S^2} W_0^2  d S^2  -\frac{1}{60} \int _{S^2} |\alpha|^2 dS^2 + \frac{11}{45} \int _{S^2} |\beta|^2 dS^2.
\end{split}
\end{equation}
\end{proof}

\section{Computing the reference Hamiltonian}\label{sec_ref_ham}
In this section, we compute the limit of the second integral in equation \eqref{energy_expression}:
\[ \int_{\Sigma_r} \tau div_{\sigma} \alpha_{H_0} d \Sigma_r.\]
For simplicity, we denote $\bar W_{0m0n}$ by $D_{mn}$.
\begin{pro} \label{pv0}
\[    \lim_{r \to 0} r^{-5 }\int_{\Sigma_r} \tau div_{\sigma} \alpha_{H_0} d \Sigma_r
=  \frac{4}{3} \int _{S^2}a^0 W_0^2 dS^2 - 10  \frac{a^ia^j}{a^0}\int _{S^2}\tilde X^i W_0 P_j dS^2.
\]
\end{pro}
\begin{proof}
We use the optimal embedding equation for the image of the isometric embedding of $\Sigma_r$ to compute the integral. 
The equation reads
\begin{equation}
div_{\sigma} \alpha_{H_0} = -(\widehat{H}\hat{\sigma}^{ab} -\hat{\sigma}^{ac} \hat{\sigma}^{bd} \hat{h}_{cd})\frac{\nabla_b\nabla_a \tau'}{\sqrt{1+|\nabla\tau'|^2}}+ div_\sigma (\frac{\nabla\tau'}{\sqrt{1+|\nabla\tau'|^2}} \cosh\theta_0|{H_0}|-\nabla\theta_0)
\end{equation}
where 
\[
 \tau' = r^3 X_0^{(3)} + r^4 X_0^{(4)} +r^5 X_0^{(5)} +O(r^6) \text{  and  }
 \sinh \theta _0= \frac{- \Delta \tau'}{|H_0| \sqrt{1+|\nabla\tau'|^2}}.
\]
It suffices to compute the expansion of $div_\sigma \alpha_{H_0}$ up to $O(r^3)$ error terms. For this purpose, we can approximate $\sqrt{1+|\nabla\tau'|^2} $ by $1$ 
and $\theta_0$ by $\frac{- \Delta \tau'}{|H_0|} $. We have
\[div \alpha_{H_0}= -[\widehat{H}\hat{\sigma}^{ab} -\hat{\sigma}^{ac} \hat{\sigma}^{bd}(\frac{\widehat{ H}}{2} \sigma_{cd}  + r^3 \AA_{cd}^{'(3)}) ] \nabla_b\nabla_a \tau'+ div_\sigma ( |{H_0}| \nabla\tau' ) + \Delta(\frac{ \Delta \tau'}{|H_0|}) +O(r^3). \]
Recall
\[ |{H_0}| =\widehat {H}+O(r^2) = \frac{2}{r} + 2W_0r +O(r^2).  \]
We have
\begin{align*} 
div _{\sigma}\alpha_{H_0}=& r^3 \hat{\sigma}^{ac} \hat{\sigma}^{bd} \AA^{'(3)}_{cd}  \nabla_b\nabla_a \tau'+ (\frac{1}{r} +  W_0 r)  \Delta \tau' + \nabla H_0 \nabla \tau '+ \frac{r}{2} \Delta[ (1 - W_0 r^2) \Delta \tau'] +O(r^3) \\
=& \frac{1}{2} (r^{-2} \Delta)(r^{-2} \Delta +2) (X_0^{(3)} + r X_0^{(4)} +r^2 X_0^{(5)})  +\\
&r^2[\tilde{\sigma}^{ac} \tilde{\sigma}^{bd} \AA^{'(3)}_{cd}  \tilde \nabla _b  \tilde \nabla _a X_0^{(3)} + W_0 \tilde \Delta X_0^{(3)} 
 + 2 \tilde \nabla W_0 \tilde \nabla X_0^{(3)} -\frac{1}{2} \tilde \Delta ( W_0 \tilde \Delta X^{(3)}_0)] +O(r^3) . \end{align*}
We compute  
\begin{align*}
  &\int_{\Sigma_r}  (a^i \tilde X^i r + r^3 a^\alpha X_{\alpha}^{(3)}) div _{\sigma}\alpha_{H_0} d\Sigma_r \\
=&\int_{\Sigma_r}  (a^i \tilde X^i r + r^3 a^\alpha X_{\alpha}^{(3)}) \Big \{  \frac{1}{2} (r^{-2} \Delta)(r^{-2} \Delta +2) ( X_0^{(3)} + r X_0^{(4)} +r^2 X_0^{(5)})\\
  & +r^2\Big [ \AA^{'(3)}_{ab}  \tilde\nabla^b \tilde \nabla^a X_0^{(3)} + 2W_0 \tilde \Delta X_0^{(3)} 
+ 2 \tilde \nabla W_0 \tilde \nabla X_0^{(3)}  -\frac{1}{2} (\tilde \Delta +2)( W_0 \tilde \Delta X^{(3)}_0)
\Big ] \Big \}  d\Sigma_r   +O(r^6) \\
=& r^5  \int_{S^2}  \frac{-a^i S_i}{2} \tilde \Delta X_0^{(3)} +\frac{1}{2}a^\alpha X_{\alpha}^{(3)} \tilde \Delta (\tilde \Delta+2)X_0^{(3)} + 
 a^i \tilde X^i ( \AA^{'(3)}_{ab}  \tilde\nabla^b \tilde \nabla^a X_0^{(3)} \\
  &+ 2W_0 \tilde \Delta X_0^{(3)} + 2 \tilde \nabla W_0 \tilde \nabla X_0^{(3)})dS^2 
  + O(r^{6}).
\end{align*}
Using Lemma \ref{WPinner}, Lemma \ref{harmonic_Sj}, Lemma \ref{xi3}  and  Lemma \ref{x03},
\begin{align*}  
  \int _{S^2} \frac{-a^i S_i}{2} \tilde \Delta X_0^{(3)}  dS^2
= &\int _{S^2} \frac{1}{3}(-4a^iD_{in}\tilde X^n + 4 a^iW_0 \tilde X^i + 4 a^i W_i) (-W_0 + 6 \frac{a^j P_j}{a^0}) dS^2\\
=& 8 \frac{a^ia^j}{a^0}\int_{S^2} W_0 \tilde X^i P_j dS^2 -\frac{4}{3} a^i \int_{S^2} W_iW_0 dS^2,
\end{align*}
\begin{align*}  
    \int _{S^2} \frac{1}{2}a^\alpha X_{\alpha}^{(3)} \tilde \Delta (\tilde \Delta+2) X_0^{(3)} dS^2
 =& \int _{S^2} (\frac{2}{5}a^i D_{in} \tilde X^n - \frac{1}{3} a^i W_i -\frac{1}{3} a^0 W_0)( -4W_0  +60 \frac{a^j P_j}{a^0} ) dS^2 \\
  =&\frac{4}{3} \int _{S^2} a^i W_iW_0  dS^2+\frac{4}{3} \int_{S^2} a^0 W_0^2   dS^2,
\end{align*}
and
\begin{align*}
   \int_{S^2} a^i \tilde X^i ( 2W_0 \tilde \Delta X_0^{(3)} + 2 \tilde \nabla W_0 \tilde \nabla X_0^{(3)})  dS^2 
=& -2  \frac{a^i a^j}{a^0}  \int_{S^2}  W_0  \tilde \nabla \tilde X^i  \tilde \nabla P_j dS^2  \\
=& - 8 \frac{a^i a^j}{a^0}  \int_{S^2}  W_0   \tilde X^i   P_j dS^2.
\end{align*}

Finally, we compute $\int a^i \tilde X^i \tilde{\sigma}^{ac} \tilde{\sigma}^{bd} \AA^{'(3)}_{cd}  \tilde \nabla _b  \tilde \nabla _a X_0^{(3)} .$ 
From Lemma \ref{x03}, we have
\[
\begin{split}
\int_{S^2} a^i \tilde X^i  \AA^{'(3)}_{ab}  \tilde \nabla ^b  \tilde \nabla^a X_0^{(3)}  dS^2 
=& \frac{a^ia^j}{a_0} \int_{S^2}\tilde X^i\AA^{'(3)}_{ab}  \tilde \nabla ^b  \tilde \nabla ^a P_j  dS^2\\
= & \frac{a^ia^j}{a_0} \int_{S^2}[\tilde  \nabla ^b  \tilde \nabla ^a \tilde X^i  \AA^{'(3)}_{ab} + 2 \tilde \nabla ^a\tilde X^i  \tilde  \nabla ^b \AA^{'(3)}_{ab}+ \tilde X^i \tilde  \nabla ^b  \tilde \nabla ^a\AA^{'(3)}_{ab}   ]   P_j   dS^2.\
\end{split}
\]
The first term vanishes since $ \AA^{'(3)}_{ab}$ is traceless. For the second and third terms, we use
\[   \tilde  \nabla ^b \AA^{'(3)}_{ab} =  \tilde  \nabla _a W_0 \]
which can be derived from the Codazzi equation. As a result, 
\[ 
\begin{split}
\int_{S^2} a^i \tilde X^i \AA^{'(3)}_{ab}  \tilde \nabla ^b  \tilde \nabla ^a X_0^{(3)}  dS^2 =& \frac{a^ia^j}{a_0} \int_{S^2}  (2 \tilde \nabla ^a\tilde X^i  \tilde  \nabla _ a W_0 -6 \tilde X^i W_0   )   P_j   dS^2 \\
=&  -10 \frac{a^ia^j}{a_0} \int_{S^2}  \tilde X^i W_0  P_j  dS^2,
\end{split}
\] 
where we apply Lemma \ref{WPinner} and integration by parts for the last equality.
\end{proof}
\section{Computing the Physical Hamiltonian}\label{sec_phy_ham}

In this section, we compute the limit of the third integral in equation \eqref{energy_expression}:
\[ \int_{\Sigma_r} \tau div_{\sigma} \alpha_{H} d \Sigma_r.\]
\begin{pro}\label{pvr}
\begin{align*}  
    \lim_{r \to \infty} r^{-5 } \int_{\Sigma_r } \tau div_{\sigma} \alpha_H d \Sigma_r 
  =  \int_{S^2} \left[\frac{4a^0}{3}W_0^2 +\frac{2}{3} a^iW_iW_0- (a^i \tilde {X}^i) |\beta|^2 \right]dS^2.
\end{align*}
\end{pro}
\begin{proof}
Recall that 
\[  div _{\sigma}  \alpha_H =   -\frac{1}{2} \Delta \ln ( -\sigma^{ab}l_{ab}) + \frac{1}{2} \Delta \ln ( \sigma^{ab} n_{ab}) - div_{\sigma} \eta.\]
The expansions for $\sigma^{ab}l_{ab}$, $ \sigma^{ab} n_{ab}$ and $ div _{\sigma} \eta$ are obtained in Lemma \ref{data}. First we compute 
\begin{align*}  
\int_{\Sigma_r}  \tau \Delta \ln (-\sigma^{ab}l_{ab}) d\Sigma_r = &\int_{S^2}  (a^i \tilde {X}^i) r  \Delta [\frac{2}{r} -\frac{ r^3}{45} |\alpha|^2  ] r^2 dS^2 +O(r^6) \\
=& -r^5 \int_{S^2}  (a^i \tilde {X}^i)   \tilde \Delta \frac{  1}{90} |\alpha|^2   dS^2+O(r^6)  \\
=& \frac{1}{45}r^5 \int_{S^2}  (a^i \tilde {X}^i)  |\alpha|^2  dS^2+O(r^6).  
 \end{align*}
 Next we compute the term involving $div_\sigma \eta$. From equation \eqref{divergence_null_connection}, we have
 \begin{equation*}  
 \begin{split}
    \int_{\Sigma_r}  \tau div_\sigma \eta d\Sigma_r  
 = & \int_{S^2 }  (  a^i \tilde {X}^ i r + r^3 a^{\beta} X^{(3)}_{\beta} )\{ \rho+ r D \rho + r^2 [\frac{ D^2 \rho}{2} +\frac{ |\alpha|^2-8|\beta|^2}{15}] \} r^2 dS^2 + O(r^6) \\
=  &r^5 \int_{S^2}\{ (  a^i \tilde {X}^ i)\frac{1}{15}( |\alpha|^2 -8|\beta|^2)+   a^{\beta} X^{(3)}_{\beta} \rho\}  dS^2 + O(r^6),
  \end{split} 
 \end{equation*}
where Lemma \ref{D_divergence} is used in the last equality. We compute
\[
  \int_{S^2} a^{\beta} X_{\beta}^{(3)} W_0  dS^2
=  \int_{S^2} (-\frac{a^0}{3}W_0- \frac{a^i}{3}W_i) W_0 dS^2.
\]
Lastly, we compute  $\int_{\Sigma_r}\tau \Delta \ln  (\sigma^{ab}n_{ab}) d\Sigma_r $. 
\begin{align*}
 &\int_{\Sigma_r }\tau \Delta \ln (\sigma^{ab}n_{ab}) d\Sigma_r \\
= &  \int_{\Sigma_r }\Delta[ ra^i \tilde {X}^i+r^3 a^{\beta} X_{\beta}^{(3)} ] \ln (1+ r^2 (\sigma^{ab}n_{ab})^{(1)}+ r^3 (\sigma^{ab}n_{ab})^{(2)}+ r^4 (\sigma^{ab}n_{ab})^{(3)}) d\Sigma_r  +O(r^6)\\
=& \int_{S^2} [-2a^i \tilde {X}^i r +r^3(2 a^0W_0+2a^i W_i - a^iS_i - \frac{4}{5}a^i D_{ij}X^j  ) ] \{ r^2 (\sigma^{ab}n_{ab})^{(1)}+ r^3 (\sigma^{ab}n_{ab})^{(2)}\\
  &\qquad+ r^4 [(\sigma^{ab}n_{ab})^{(3)} -\frac{1}{2}((\sigma^{ab}n_{ab})^{(1)})^2] \} dS^2 +O(r^6) .
\end{align*}
Using Lemma \ref{data} for the expansion of  $\sigma^{ab}n_{ab}$ and applying Lemma \ref{D_divergence}, we conclude
\begin{align*} &\int_{\Sigma_r}\tau \Delta \ln (\sigma^{ab}n_{ab})  d \Sigma_r \\
=   & r^5 \int_{S^2} \{( 2a^i W_i + 2 a^0W_0  - a^iS_i - \frac{4}{5}a^i D_{ij}\tilde{X}^j )W_0  -2a^i \tilde {X}^i   [(\sigma^{ab}n_{ab})^{(3)} -\frac{1}{2}W_0^2] \}dS^2  +O(r^6) \\
=   & r^5 \int_{S^2} [( 2a^i W_i + 2 a^0W_0-  a^iS_i  )W_0 + a^i \tilde {X}^i (\frac{22}{45} |\beta|^2-\frac{1}{15}|\alpha|^2) ]dS^2 +O(r^6) .
\end{align*}
The proposition follows from collecting terms and equation \eqref{relation_alpha_beta_2}.
\end{proof}
\section{Evaluating the energy}\label{sec_eva_energy}
We now evaluate the energy for an observer $T_0= (\sqrt{1+|a|^2}, -a_1,-a_2,-a_3)$ using Lemma \ref{lemmaenergy} and Proposition \ref{pe2}, \ref{pv0} and \ref{pvr}. It is easy to observe that the energy takes the following form
\[ \sum_{\alpha}A_\alpha a^\alpha + \sum_{ij} A_{ij}\frac{a^ia^j}{a^0}.  \]
The following lemma, which follows from Lemma \ref{sphere}, is useful for evaluating the above integrals. Recall that we denote $\bar W_{0m0n}$ by $D_{mn}$.
\begin{lemma}  \label{sphereintegral}
\begin{align}
\int_{S^2} W_0 \tilde X ^i \tilde X^j dS^2&= \frac{8 \pi}{15}D_{ij}\\
\int_{S^2} W_0^2 dS^2&=\frac{8 \pi}{15} \sum_{ij} D_{ij}^2\\
\int_{S^2} W_i \tilde X^j \tilde X^l dS^2 &=\frac{4 \pi}{15}(\bar W_{0lij}+\bar W_{0jil})\\
\int_{S^2} W_0 W_i dS^2&= \frac{8 \pi}{15}\sum_{jl} D_{jl}\bar W_{0lij}.
\end{align}
\end{lemma}
\begin{proof}
These follow from the definitions of $W_0$ and $W_i$ as well as Lemma \ref{sphere}.
\end{proof}
First we compute $A_0$. Collecting the coefficients, we have
\begin{align*}
A_0 =& \frac{1}{12} \int_{S^2} W_0^2 dS^2 -\frac{1}{60} \int _{S^2} |\alpha|^2 dS^2 + \frac{11}{45} \int _{S^2} |\beta|^2 dS^2\\
=& \frac{1}{12} \int_{S^2} W_0^2 dS^2 + \frac{1}{9} \int _{S^2} |\beta|^2 dS^2.
\end{align*}
Equation \eqref{relation_alpha_beta_1} is used in the last equality.
\begin{lemma}
\[ \int_{S^2} |\beta|^2  dS^2  = \frac{12 \pi}{15} \sum D_{ij}^2  + \frac{6 \pi}{15}  \sum \bar W_{ijk0}^2  .\]
\end{lemma}
\begin{proof}
\begin{align*}
  \int_{S^2} |\beta|^2 dS^2 
=&\int_{S^2}\bar  W_{L i  r 0} \bar   W_{L i r 0} dS^2 - \int_{S^2}\bar    W_{0 r  r 0}\bar   W_{0 r r 0} dS^2 \\
=&\int_{S^2} \bar  W_{0 i r 0}\bar  W_{0 i r 0}  dS^2 + \int_{S^2} \bar  W_{r i r 0}\bar  W_{r i r 0} dS^2 - \int_{S^2} W_0^2dS^2  \\ 
=& \frac{20 \pi}{15} D_{ij}D_{ij} + \frac{4 \pi}{15}\bar   W_{jik0}\bar  W_{min0} (\delta_{jk}\delta_{mn} +\delta_{jm}\delta_{kn}+\delta_{jn}\delta_{mk} )  - \frac{8 \pi}{15}D_{ij}D_{ij} \\
=& \frac{12 \pi}{15} D_{ij}D_{ij} + \frac{4 \pi}{15}\bar   W_{jik0}(\bar  W_{jik0} + \bar  W_{kij0}) ,
\end{align*}
where Lemma \ref{sphereintegral} is used in the second last equality. The lemma follows from the first Bianchi identity, since
\[\bar   W_{ijk0}\bar  W_{ijk0} = - \bar  W_{ijk0}(\bar  W_{jki0}+ \bar  W_{kij0}) = 2 \bar  W_{jik0} \bar  W_{kij0}. \]
\end{proof}
Thus, we have proved that
\begin{pro} \label{lemmaa^0}
\[A_0 = \frac{4 \pi}{15}( \frac{1}{6} \sum_{ijk}\bar  W_{0ijk}^2 +  \frac{1}{2}\sum_{ij} D_{ij}^2 ).\]
\end{pro}
Next we compute $A_i$. Collecting the coefficients, we have
\begin{equation}\label{Aifirst} 
\begin{split}
A_i = \int _{S^2} \tilde X^i |\beta|^2  dS^2-  \frac{2}{3}\int_{S^2} W_i W_0 dS^2.
\end{split}
\end{equation}
We compute
\begin{align*}
   \int_{S^2}\tilde  X^i  |\beta|^2 dS^2
=& \int_{S^2}\tilde  X^i \bar  W_{Ljr0}\bar   W_{Lkr0} (\delta^{jk}- X^kX^j) dS^2 \\
=&\int_{S^2} \tilde  X^i \bar  W_{Ljr0} \bar   W_{Ljr0} dS^2 - \int_{S^2} \tilde  X^i \bar  W_{0r r0} \bar  W_{0rr0} dS^2 \\
=& \frac{8 \pi}{15} D_{jm} \bar  W_{0mij}.
\end{align*}

As a result, we have 
\begin{pro} \label{lemmaa^i}
\[
A_i= \frac{4 \pi}{15}(\frac{2}{3}D_{jm} W_{0mij} ).\]
\end{pro}

At the end, we have

\begin{pro}\label{proa^ij}
\begin{equation}
A_{ij} = -\frac{2 \pi}{45} \delta_{ij} \sum_{m,n} D_{mn}^2.
\end{equation}
\end{pro}

\begin{proof}

To compute $A_{ij}$, we combine Lemma \ref{lemmaenergy} and Proposition \ref{pe2}, \ref{pv0} and \ref{pvr} and derive
\begin{equation}
\begin{split}
A_{ij}=&\int_{S^2}  \frac{W_0}{2}[R_{ij}+2 \tilde \nabla \tilde{X}^i \cdot \tilde \nabla (X_j^{(3)}+ P_j)+\tilde{X^i}(S_j-\tilde{\Delta}(X_j^{(3)}+12P_j)] dS^2 \\
 &-\frac{3}{4 } \int _{S^2}W_0^2\tilde{X}^i\tilde{X}^j dS^2   -30 \int _{S^2} P_i P_j  dS^2 -10  \int_{S^2}\tilde X^i W_0P_j  dS^2.\end{split}
\end{equation}
We compute
\[
\begin{split}
&\int _{S^2} \frac{W_0}{2}[R_{ij}+2 \tilde \nabla \tilde{X}^i \cdot \tilde \nabla (X_j^{(3)} + P_j)+\tilde{X^i}(S_j - \tilde{\Delta}X_j^{(3)} +12P_j)] dS^2\\
=&\int _{S^2}  \frac{W_0}{6} [ 2\tilde {X}^i\tilde {X}^k \bar{W}_{0j0k}
 + 2\tilde {X}^j\tilde {X}^k\bar  W_{0k0i}  -2\bar  W_{0i0j}-W_0 \delta_{ij} -W_0 \tilde {X}^i\tilde {X}^j] dS^2 \\
&+ \int _{S^2} W_0  \tilde \nabla \tilde{X}^i \cdot \tilde \nabla (\frac{2}{5}\bar  W_{0j0n}\tilde X^ n) +\int  _{S^2} W_0 \tilde{X^i}(\frac{2}{3} \tilde {X}^j W_0-\frac{2}{3} \bar   W_{0j0n} \tilde {X}^n+\frac{2}{5}\bar  W_{0j0n} \tilde X^n ) dS^2 \\
= &\int _{S^2}  \frac{W_0}{6} [ 2\tilde {X}^i\tilde {X}^k \bar   W_{0j0k}
 + 2\tilde {X}^j\tilde {X}^k\bar  W_{0k0i}  -2\bar  W_{0i0j}-W_0 \delta_{ij} -W_0 \tilde {X}^i\tilde {X}^j]dS^2 \\
&+\int  _{S^2}W_0 \tilde{X^i}(\frac{2}{3} \tilde {X}^j W_0-\frac{2}{3}  \bar  W_{0j0n} \tilde {X}^n)dS^2 \\
= &\frac{1}{2}\int _{S^2} W_0^2 \tilde{X^i} \tilde {X}^j dS^2 - \frac{1}{6}  \delta_{ij} \int_{S^2} W_0^2 dS^2.
\end{split}
\]
Moreover, 
\[
\begin{split}
  -30 \int_{S^2} P_i P_j dS^2 =  & 5 \int_{S^2} W_0 \tilde X^i P_j dS^2 \\
= &  \frac{1}{3}\int_{S^2} W_0\bar  W_{0j0n} \tilde X^i  \tilde X^n dS^2 - \frac{5}{6} \int_{S^2} W_0^2  \tilde X^i  \tilde X^j dS^2,
\end{split}
 \] and thus 
\[A_{ij}=  \frac{7}{12}\int_{S^2} W_0^2 \tilde{X^i} \tilde {X}^j  dS^2- \frac{1}{6}  \delta_{ij} \int _{S^2}W_0^2 dS^2 -  \frac{1}{3}\int_{S^2} W_0\bar  W_{0j0n} \tilde X^i  \tilde X^n  dS^2.\]
Applying the last formula in Lemma \ref{sphere}, we obtain
\[\int_{S^2} W_0^2 \tilde {X}^m \tilde {X}^n  dS^2=\frac{4\pi}{3\cdot 5\cdot 7}(2\delta_{mn} \sum_{ij} D_{ij}^2+ 8  \sum_i D_{im}D_{in}).\]

\end{proof}
\section{Minimizing the energy}\label{sec_energy_min}
In this section, we study the existence and uniqueness of observers $T_0= (a^0, -a^1,-a^2,-a^3)$ that minimize the energy computed in Section 10. In Lemma \ref{x03}, we solve the leading order term of the optimal embedding equation. In this section, we show that by minimizing the quasi-local energy, the optimal embedding equation can be solved to higher orders. 

Throughout the section, we denote the surface $\Sigma_r$ simply by $\Sigma$ and the isometric embedding $X_r$ by $X$.
We consider pairs $(X,T_0)$ with the following expansion:
\begin{equation}\label{assume1}
\begin{split} 
X_0  = & \sum_{i=3}^{\infty}X_0^{(i)}r^i \\
X_k = &  r \tilde X^k+ \sum_{i=3}^{\infty}X_k^{(i)}r^i  \\
T_0= & (a^0,-a^i)+\sum_{i=1}^{\infty}T_0^{(i)}r^i. 
\end{split}
\end{equation}
For such pairs, we have 
\[ E(\Sigma,X,T_0) = \sum_{i=5}^{\infty} E(\Sigma,X,T_0)^{(i)}r^i.  \]
Similar to \cite{Chen-Wang-Yau1}, the $O(r^k)$ part of the optimal embedding equation is of the form
\[  \frac{1}{2} \tilde \Delta(\tilde \Delta+2) X_0^{(k+3)} = M_k(X_0^{(3)}, \dots, X_0^{(k+2)}, T_0^{(0)},\dots, T_0^{(k)}).\]
The equation is solvable if 
\[  \int_{S^2}M_k( X_0^{(3)}, \dots, X_0^{(k+2)}, T_0^{(0)},\dots, T_0^{(k)}) \tilde {X}^i d S^2 =0. \]
We need the following lemma about the above integral.
\begin{lemma}\label{indepedent2}
\[  \int_{S^2}M_k(X_0^{(3)}, \dots, X_0^{(k+2)}, T_0^{(0)},\dots, T_0^{(k)}) \tilde {X}^i d S^2 \]
is  independent of $X_0^{(k+2)}$ and $T_0^{(j)}$ for $j \ge k-1$.
\end{lemma}
\begin{proof}
The optimal embedding equation reads
\[ div_\sigma (f \nabla \tau) - \Delta( \sinh^{-1}(\frac{f \Delta \tau}{|H||H_0|})) = div_{\sigma}\alpha_{H_0} - div_{\sigma}\alpha_{H}.\]
For the right hand side, $ div_{\sigma}\alpha_{H}$ is independent of $(X,T_0)$.  $div_{\sigma}\alpha_{H_0}$ is independent of $T_0$. While it depends on $X$, $X_0^{(k+2)}$ contributes 
\[ \frac{1}{2} \tilde \Delta(\tilde \Delta+2) X_0^{(k+2)} r^{k-1} +O(r^{k+1}) \]
to the right hand side and does not contribute to $M_k$.
For the left hand side, $|H|$ is independent of $(X,T_0)$.  $X_0^{(k+2)}$ only contributes to $|H_0|$ by terms of the order $O(r^{k+2})$ and does not contribute to $M_k$. For $T_0^{(k-1)}=(b^0,-b^i)$, it contributes to $\tau$ by
\[  b^i \tilde X^i r^{k-1} + O(r^{k}). \]
As in the proof of Lemma \ref{x03}, this contributes to the left hand side by a linear combination of $-12$-eigenfunctions.
\end{proof}
\begin{remark}
Here are some examples to illustrate the above lemma. Set $k=1$, the $O(r)$ order of the optimal embedding equation is
\begin{equation} \label{eq_r1}  
\frac{1}{2} \tilde \Delta(\tilde \Delta+2) X_0^{(4)} = M_1(X_0^{(3)}, T_0^{(0)}, T_0^{(1)}).
\end{equation}
By the lemma, it is solvable for any choice of $X_0^{(3)}$, $T_0^{(0)}$ and $T_0^{(1)}$.  On the other hand, set $k=2$, the  $O(r^2)$ order of the optimal embedding equation is
\begin{equation} \label{eq_r2} 
 \frac{1}{2} \tilde \Delta(\tilde \Delta+2) X_0^{(5)} = M_2( X_0^{(3)}, X_0^{(4)}, T_0^{(0)}, T_0^{(1)},T_0^{(2)}).
\end{equation}
Its solvability depends only on $T_0^{(0)}$ and $X_0^{(3)}$.  In fact, it only depends on the choice of $T_0^{(0)}$ since  $X_0^{(3)}$ is determined by $T_0^{(0)}$ using Lemma \ref{x03}.
\end{remark}
Let $E(\Sigma,X(T_0),T_0)$ be the quasi-local energy of $\Sigma$ with embedding $X(T_0)$ and observer $T_0$ where $X(T_0)$ is the isometric embedding of $\Sigma$ into $\R^{3,1}$ determined by Lemma \ref{xi3} and Lemma \ref{x03}. 
We have
\[ E(\Sigma,X(T_0),T_0) = \sum_{i=5}^{\infty} E(\Sigma,X(T_0),T_0)^{(i)}r^i , \]
where
\[
 E(\Sigma,X(T_0),T_0)^{(5)} =\frac{ 1}{90}  \{  (\frac{1}{2}\sum_{k, m, n} \bar W_{0kmn}^2+  \sum_{m, n} \bar W_{0m0n}^2 )a^0  + 2\sum_i \sum_{m, n} \bar  W_{0m0n}\bar  W_{0min} a^i + \sum_{m, n} \frac{\bar W_{0m0n}^2 }{2a^0} \}.
\] 
We show that generically, there is a unique minimizer $T_0$ of $ E(\Sigma,X(T_0),T_0)^{(5)}$. Moreover, for the minimizer, equation \eqref{eq_r2}  is solvable. We start with the following lemma.
\begin{lemma}
Let $V=(\frac{1}{2}\sum \bar{W}_{0kmn}^2+  \sum \bar{W}_{0m0n}^2, 2\sum \bar{W}_{0m0n}\bar{W}_{0min})$. $V$ is future directed non-spacelike. Moreover, $V$ is timelike unless in some orthonormal frame, we have
\begin{equation} \label{null_condition} \left( D_{ij} \right)= \left( \begin{array}{ccc}
0 & 0 & 0 \\
0 & b & 0 \\
0 & 0 & -b \end{array} \right)   \qquad    \left( E_{ij} \right)= \left( \begin{array}{ccc}
0 & 0 & 0 \\
0 & 0 & b \\
0 & b & 0 \end{array} \right)  \end{equation}
where
\[ \bar W_{0ijk} =\epsilon^{jkn} E_{in}.  \]
\end{lemma}
\begin{remark}
$D$ and $E$ are both symmetric traceless 3 by 3 matrices. Together, they capture all the ten independent components of the Weyl curvature at a point.
\end{remark}
\begin{proof}
It suffices to show that 
\begin{equation} \label{eq_timelike}
(\frac{1}{2}\sum \bar W_{0kmn}^2+  \sum \bar W_{0m0n}^2)^2  \ge  4 \sum_i  (\sum_{m,n} \bar W_{0m0n}\bar W_{0min})^2. 
\end{equation}
We pick an orthonormal frame which diagonalizes $D$. Suppose 
\[ \left( D_{ij} \right)= \left( \begin{array}{ccc}
a & 0 & 0 \\
0 & b & 0 \\
0 & 0 & -(a+b)\end{array} \right).   \]
In this case, the only components of $E$ that appear on the right hand side are the off-diagonal entries. Let
\[\bar  W_{0131} =c   \qquad\bar  W_{0212} =d  \qquad \bar W_{0121}=e\]
and the inequality follows from
\[ (c^2+d^2+e^2 + a^2+b^2+ab)^2  - [(a-b)^2c^2+(2b+a)^2 d^2 + (2a+b)^2 e^2] \ge 0, \]
where we discard terms of the form $W_{0123}$, namely the entries on the diagonal of $E$, from the left hand side of equation (\ref{eq_timelike}). 
However
\begin{align*}
   &(c^2+d^2+e^2 + a^2+b^2+ab)^2  - ((a-b)^2c^2+(2b+a)^2 d^2 + (2a+b)^2 e^2) \\
= & c^2(a+b)^2 + (c^2+ab)^2 + d^2a^2+ (d^2-b^2-ab)^2 +e^2b^2+(e^2-a^2-ab)^2. \end{align*}
Hence the vector $(\frac{1}{2}\sum \bar W_{0kmn}^2+  \sum\bar  W_{0m0n}^2,2 \sum \bar W_{0m0n}\bar W_{0min})$ is future directed non-spacelike.
If the vector is indeed null, then the diagonal of $E$ is $0$ and 
$c(a+b)$, $(c^2+ab)$, $da$, $(d^2-b^2-ab)$, $eb$ and $(e^2-a^2-ab)$ must all vanish. Then, up to switching the indices, $D_{ij}$ and $E_{ij}$ must 
be the ones indicated in the statement of the lemma. 
\end{proof}
\begin{cor}
The energy functional $E(\Sigma,X(T_0),T_0)^{(5)}$ is non-negative. Moreover, it is positive and proper when $V$ is timelike.
\end{cor}
Hence, when $V$ is timelike, there is at least one observer $T_0=(a^0,-a^i)$
which minimizes $E(\Sigma,X(T_0),T_0)^{(5)}$. We show that under the same condition, the minimizer is unique. 
\begin{lemma}\label{minimize_direction}
Assume  $V$ is timelike then there is a unique $T_0=(\bar a^0, -\bar a^i)$ that minimizes $E(\Sigma,X(T_0),T_0)^{(5)}$.

\end{lemma}
\begin{proof}
It suffices to show that $E(\Sigma,X(T_0),T_0)^{(5)}$ is a strictly convex function of $(a^1,a^2,a^3)$ since a convex function cannot
have two critical points.  Recall $E(\Sigma,X(T_0),T_0)^{(5)}$ is 
\[
\frac{1}{90}  \{  (\frac{1}{2}\sum \bar W_{0kmn}^2+  \sum \bar W_{0m0n}^2 )a^0  + 2\sum \bar W_{0m0n}\bar W_{0min} a^i + \frac{1}{2}\sum \bar W_{0m0n}^2 \frac{1}{a^0} \} .\]
Since $\partial_{a^i}a^0= \frac{a^i}{a^0}$, the first derivative of $90E(\Sigma,X(T_0),T_0)^{(5)}$ with respect to $a^i$ is 
\[  (\frac{1}{2}\sum \bar W_{0kmn}^2+  \sum \bar W_{0m0n}^2 )\frac{a^i}{a^0}  + 2\sum \bar W_{0m0n}\bar W_{0min} - \frac{1}{2}\sum \bar W_{0m0n}^2 \frac{a^i}{(a^0)^3} \]
and the second derivative is 
\begin{align*}   
 &(\frac{1}{2}\sum\bar  W_{0kmn}^2+  \sum \bar  W_{0m0n}^2 )(\frac{1}{a^0} - \frac{(a^i)^2}{(a^0)^3})
 - \frac{1}{2}\sum \bar  W_{0m0n}^2( \frac{1}{(a^0)^3} - \frac{3(a^i)^2}{(a^0)^5}) \\
=&(\frac{1}{2}\sum\bar  W_{0kmn}^2+  \sum\bar  W_{0m0n}^2 )(\frac{1}{a^0} - \frac{(a^i)^2}{(a^0)^3})
 - \frac{1}{2}\sum \bar W_{0m0n}^2( \frac{1}{(a^0)^3} - \frac{ (a^i)^2}{(a^0)^5}) + \frac{\sum \bar  W_{0m0n}^2 (a^i)^2}{(a^0)^5} \\
\ge &(\frac{1}{2}\sum \bar W_{0kmn}^2+ \frac{1}{2} \sum \bar  W_{0m0n}^2 )(\frac{1}{a^0} - \frac{(a^i)^2}{(a^0)^3})
 \end{align*}
This is positive unless the Weyl curvature tensor vanishes at $p$. 
\end{proof}
As a result, there is a unique observer $\bar T_0 =(\bar a^0, -\bar a^i)$ such that for any other $T_0$, 
\[E(\Sigma,X(T_0),T_0)^{(5)} \ge E(\Sigma,X(\bar T_0),\bar T_0)^{(5)}.  \]
\begin{lemma}\label{lemma_minimizing}
For every pair $(X,T_0)$ with expansion given in equation \eqref{assume},
\[E(\Sigma,X,T_0)^{(5)} \ge E(\Sigma,X(\bar T_0), \bar T_0)^{(5)}.  \]
\end{lemma}
\begin{proof}
It suffices to show that 
\[E(\Sigma,X,T_0)^{(5)} \ge E(\Sigma,X(T_0), T_0)^{(5)}.  \]
However, tracing through the dependence of $E(\Sigma,X,T_0)^{(5)}$ on $X$, we have
\[ E(\Sigma,X,T_0)^{(5)} - E(\Sigma,X(T_0), T_0)^{(5)} = M( X_0^{(3)}) - M(X^{(3)}_0(T_0)) , \] 
where
\[ M(f) = \int_{S^2} [\frac{1}{4}f \tilde \Delta(\tilde \Delta +2)f  + fg] dS^2  \]
for some function $g$ on $S^2$. This is a convex functional of $f$ and the $X^{(3)}_0(T_0)$ obtained in Lemma \ref{x03} is its critical point.
\end{proof}
Finally, we show that equation \eqref{eq_r2} is solvable for the minimizer  $(\bar a^0, -\bar a^i)$.  We may assume that we pick some $T_0^{(1)},T_0^{(2)}$,  $X_0^{(3)}$ and $X_0^{(4)}$ such that the top order term of the optimal embedding equation and  equation \eqref{eq_r1} are solved. From Lemma \ref{lemma_minimizing}, we have
\[ \partial_{a^i} E(\Sigma,X(\bar T_0),T_0)^{(5)} =0.\] 
However, 
\[  \partial_{a^i}  \tau = \tilde X^i r +O(r^2). \] 
We conclude that
 \[ \partial_{a^i} E(\Sigma,X(\bar T_0),T_0)^{(5)} = \pm \int_{S^2} \tilde X^i [ \frac{1}{2} \tilde \Delta(\tilde \Delta+2) X_0^{(5)}- M_2(X_0^{(3)},X_0^{(4)}, T_0^{(0)}, T_0^{(1)},T_0^{(2)}) ] dS^2. \] 
As a result, 
\[\int_{S^2} \tilde X^i M_2( X_0^{(3)},X_0^{(4)}, T_0^{(0)}, T_0^{(1)},T_0^{(2)})  dS^2=0 . \]

\end{document}